\documentclass[a4paper, 11pt, english]{article}

\usepackage{lscape}
\usepackage{pdflscape}

\usepackage{stmaryrd}

\usepackage{cmll}

\usepackage{amscd}
\usepackage{amssymb,amsmath,amsthm}
\usepackage{mathptmx}
\usepackage{mathrsfs}
\usepackage{color}
\usepackage{xspace}
\usepackage{bussproofs}
\EnableBpAbbreviations

\usepackage{tikz}

  \usepackage{txfonts}

\usepackage{bpextra}

\usepackage{enumerate}

\usepackage{centernot}

\usepackage{mathtools}

\usepackage{hyperref}

\usepackage{cleveref}

\usepackage{relsize}

\usepackage{caption}

\usepackage{multirow}

\usepackage{multicol}

\usepackage{array}
\newcolumntype{C}[1]{>{\centering\arraybackslash}p{#1}}
\newcolumntype{L}[1]{>{\arraybackslash}p{#1}}

\usepackage{comment}
\usepackage{rotating}

\newtheorem{theorem}{Theorem}[section]

\newtheorem{proposition}[theorem]{Proposition}

\newtheorem{definition}[theorem]{Definition}

\def\fCenter{{\mbox{$\ \vdash\ $}}}

\newcommand{\fns}{\footnotesize}

\newcommand{\marginnote}[1]{\marginpar{\raggedright\tiny{#1}}}

\def\mc{\multicolumn}

\newcommand{\mneg}{\ensuremath{^{\bot}}\xspace}

\newcommand{\ATOPBOT}{\ensuremath{\textrm{I}}\xspace}
\newcommand{\WCIRCW}{\ensuremath{\circ}\xspace}
\newcommand{\BCIRCB}{\ensuremath{\bullet}\xspace}
\newcommand{\wboxw}{\ensuremath{\Box}\xspace}
\newcommand{\wdiaw}{\ensuremath{\Diamond}\xspace}
\newcommand{\bboxb}{\ensuremath{\blacksquare}\xspace}
\newcommand{\bdiab}{\ensuremath{\Diamondblack}\xspace}

\newcommand{\andl}{\ensuremath{\wedge}\xspace}
\newcommand{\topl}{\ensuremath{\top}\xspace}
\newcommand{\orl}{\ensuremath{\vee}\xspace}
\newcommand{\botl}{\ensuremath{\bot}\xspace}
\newcommand{\rarrl}{\ensuremath{\rightarrow}\xspace}

\def\drarrl{\mbox{$\,>\mkern-7mu\raisebox{-0.065ex}{\rule[0.5865ex]{1.38ex}{0.1ex}}\,$}}

\newcommand{\neglg}{\ensuremath{\neg}\xspace}
\newcommand{\neglf}{\ensuremath{{\sim}}\xspace}
\newcommand{\ANDORL}{\ensuremath{\,,}\xspace}
\newcommand{\TOPBOTL}{\ensuremath{\textrm{I}}\xspace}

\newcommand{\NEGLG}{\ensuremath{\ast}\xspace}
\newcommand{\NEGLF}{\ensuremath{\circledast}\xspace}
\newcommand{\ANDORDL}{\ensuremath{\,;}\xspace}
\newcommand{\ARRDL}{\ensuremath{\sqsupset}\xspace}

\newcommand{\ptopp}{\ensuremath{\wp}\xspace}
\newcommand{\pbotp}{\ensuremath{\varnothing}\xspace}
\newcommand{\pandp}{\ensuremath{\cap}\xspace}

\newcommand{\porp}{\ensuremath{\cup}\xspace}

\newcommand{\PTOPBOTP}{\ensuremath{\circledS}\xspace}
\newcommand{\PANDORP}{\ensuremath{\centerdot}\xspace}
\newcommand{\PARRP}{\ensuremath{\supset}\xspace}
\newcommand{\pwdiawp}{\ensuremath{\Diamond}\xspace}
\newcommand{\pbdiabp}{\ensuremath{\Diamondblack}\xspace}
\newcommand{\pbboxbp}{\ensuremath{\blacksquare}\xspace}
\newcommand{\pwboxwp}{\ensuremath{\Box}\xspace}
\newcommand{\PWCIRCWP}{\ensuremath{\circ}\xspace}
\newcommand{\PBCIRCBP}{\ensuremath{\bullet}\xspace}

\newcommand{\ptop}{\ensuremath{\wp}\xspace}
\newcommand{\pbot}{\ensuremath{\varnothing}\xspace}
\newcommand{\pand}{\ensuremath{\cap}\xspace}
\newcommand{\prarr}{\ensuremath{\,{\raisebox{-0.065ex}{\rule[0.5865ex]{1.38ex}{0.1ex}}\mkern-5mu\supset}\,}\xspace}
\newcommand{\por}{\ensuremath{\cup}\xspace}
\newcommand{\pdrarr}{\ensuremath{\,{\supset\mkern-5.5mu\raisebox{-0.065ex}{\rule[0.5865ex]{1.38ex}{0.1ex}}}\,}\xspace}

\newcommand{\PTOPBOT}{\ensuremath{\circledS}\xspace}
\newcommand{\PANDOR}{\ensuremath{\centerdot}\xspace}
\newcommand{\PARR}{\ensuremath{\supset}\xspace}
\newcommand{\pwbox}{\ensuremath{\Box}\xspace}

\newcommand{\pbdia}{\ensuremath{\Diamondblack}\xspace}
\newcommand{\PWCIRC}{\ensuremath{\circ}\xspace}
\newcommand{\PBCIRC}{\ensuremath{\bullet}\xspace}

\newcommand{\topp}{\ensuremath{\wp^\mathrm{op}}\xspace}
\newcommand{\botp}{\ensuremath{\varnothing^\mathrm{op}}\xspace}
\newcommand{\andp}{\ensuremath{\cap^\mathrm{op}}\xspace}
\newcommand{\rarrp}{\ensuremath{\,{\raisebox{-0.065ex}{\rule[0.5865ex]{1.38ex}{0.1ex}}\mkern-5mu\supset^\mathrm{op}}\,}\xspace}
\newcommand{\orp}{\ensuremath{\cup^\mathrm{op}}\xspace}
\newcommand{\drarrp}{\ensuremath{\,{\supset\mkern-5.5mu\raisebox{-0.065ex}{\rule[0.5865ex]{1.38ex}{0.1ex}}^\mathrm{op}}\,}\xspace}

\newcommand{\TOPBOTP}{\ensuremath{\circledS^\mathrm{op}}\xspace}
\newcommand{\ANDORP}{\ensuremath{\centerdot^\mathrm{op}}\xspace}
\newcommand{\ARRP}{\ensuremath{\supset^\mathrm{op}}\xspace}
\newcommand{\diawp}{\ensuremath{\Diamond^\mathrm{op}}\xspace}

\newcommand{\boxbp}{\ensuremath{\blacksquare^\mathrm{op}}\xspace}

\newcommand{\CIRCWP}{\ensuremath{\circ^\mathrm{op}}\xspace}

\usetikzlibrary{arrows}
\usetikzlibrary{matrix}
\usetikzlibrary{patterns}
\usetikzlibrary{shapes}
\makeindex

\title{Lattice Logic Properly Displayed}

\author{Giuseppe Greco \and
Alessandra Palmigiano\thanks{This research is supported by the NWO Vidi grant 016.138.314, by the NWO Aspasia grant 015.008.054, and by a Delft Technology Fellowship awarded to the second author in 2013.}
}
\date{}

\begin{document}

\maketitle

\begin{abstract}
We introduce a {\em proper} display calculus for (non-distributive) Lattice Logic which is sound, complete, conservative, and enjoys cut-elimination and subformula property. Properness (i.e.~closure under uniform substitution of all parametric parts in rules) is the main interest and added value of the present proposal, and allows for the smoothest Belnap-style proof of cut-elimination. Our proposal builds on an algebraic and order-theoretic analysis of the semantic environment of lattice logic, and applies the guidelines of the multi-type methodology in the design of display calculi.
{\em 2010 Math.~Subj.~Class.} 03F52, 03F05, 03G10, 06A15, 06B15, 08A68, 18A40.
\end{abstract}

\tableofcontents

\section{Introduction}

 In the present paper, a proper (multi-type) display calculus is introduced  for {\em lattice logic}, by which we indicate the $\{\wedge, \vee, \top, \bot\}$-fragment of classical propositional logic without distributivity. This work is motivated by and embeds in a more general theory---that of the so-called {\em proper  multi-type calculi}, introduced in \cite{LORI,Multitype,PDL} and further developed  in \cite{Trends,LoRC,Inquisitive,GP:linear}---which aims at creating a  proof-theoretic environment designed on the basis of algebraic and order-theoretic insights, and encompassing in a uniform and modular way a very wide range of non-classical logics, spanning from logics such dynamic epistemic logic, PDL, and inquisitive logic to lattice-based substructural (modal) logics.

Proper multi-type calculi are a natural generalization of Belnap's display calculi \cite{Be82} (later refined by Wansing's notion of proper display calculi \cite{wansing}), the salient features of which they inherit. Like display calculi, proper multi-type calculi uniformly verify the assumptions of a Belnap-style cut elimination metatheorem, which guarantees that a uniform reduction strategy for cut elimination  can be applied to each of them. The uniform applicability of one and the same reduction strategy is due, both for  display calculi and proper multi-type calculi, to a neat separation of roles enforced between introduction rules for logical connectives and structural rules. Indeed,  introduction rules are defined following a very uniform and rigid design (the so-called {\em multiplicative} form)
which only allows to capture the most basic information on the polarity of each coordinate of each logical connective. The uniformity of this design  is key to  achieving a uniform formulation of the so-called `parametric step' in the cut-elimination procedure. Indeed, it is precisely what guarantees that a given application of the cut rule in which at least one cut formula is not principal can be `moved upwards', without reducing the complexity of the cut formula,  by inserting new cuts  where the parametric cut formula has been {\em introduced}. However, if all introduction rules are to verify one and the same  design, the information on the distinctive features of each individual connective must be encoded somewhere else. Encoding the behaviour specific to each connective, as well as the information about how the connectives interact, is the specific task of the structural rules. The design of the structural rules is also required to satisfy certain  {\em analyticity} conditions, the definition of which is  motivated as well by the metatheorem. The extra expressivity needed to encode the information on the specific logic purely at the structural level is guaranteed by  
a richer language which includes  {\em structural} connectives as well as {\em logical} connectives. Typically, in  display calculi, each logical connective has a structural counterpart, which encodes its behaviour 
at a purely structural level.


However, in most calculi for (general) lattice-based logics \cite{SambinBattilottiFaggian2000,NegriVonPlato2004}, including display calculi \cite{Be90}, the introduction rules for conjunction and disjunction  are given  in so-called {\em additive} form, which, unlike the multiplicative form, does not involve structural counterparts of conjunction and disjunction in its formulation. 
The reason for the non-standard treatment of conjunction and disjunction in the setting of display calculi is the following trade-off:   introducing the structural counterpart of these connectives would require the addition of certain rules (the {\em display postulates}) in order to enforce a property (the {\em display property}, from which these calculi are named) which is key to the satisfaction of one of the assumptions of the cut elimination metatheorem; however, the addition of display postulates would make it possible for the resulting calculus to derive the unwanted distributivity axioms as theorems. So, the need to block the derivation of distributivity is at the root of the non-standard design choice of having logical connectives without their structural counterpart (cf.\ \cite{Be93}).

However, as hinted above, from the point of view of the development of a general theory, this choice yields significant disadvantages. In particular, one loses the possibility of expressing the interactions between conjunction and disjunction and (possibly) other connectives  at the structural level, by means of {\em analytic} structural rules.  The remarkable property of these rules is that they can be safely and modularly added to a proper multi-type calculus so as to preserve its cut elimination theorem.  The loss in expressive power is all the more a disadvantage, because a uniform  theory of analytic extensions of proper multi-type calculi is being developed \cite{CoPa:Multi-type}, thanks to the systematic connections established in \cite{GMPTZ} between proper display calculi and  the algebraic theory of unified correspondence
\cite{CoGhPa14,CoPa12,ConPalSou,CFPS15,CoCr14,CPSZ,CPZ:Trans,ConRob,FrPaSa16,PaSoZh15,PaSoZh16,LeRoux:MThesis:2016,MZ16}
 (which is also available for substructural logics and other logics algebraically captured by general lattice expansions, cf.\ \cite{CoPa:non-dist,CP:constructive,CCPZ,CFPPTW}). These connections  have  made it possible to  characterize  the syntactic shape of axioms (the so-called {\em analytic inductive} axioms) which can be equivalently translated into analytic rules of a proper display calculus. Thus, having conjunction and disjunction as logical connectives without their structural counterpart blocks the  access to the benefits of a general and modular proof theory of analytic extensions of lattice-based logics.

The proper display calculus for the logic of  lattices discussed in the present talk enjoys the {\em full display} property,  and all its introduction rules are given in the standard {\em multiplicative} form. This is made possible thanks to the introduction of a richer, {\em multi-type} language for lattice logic which is motivated and justified semantically  by the well known  double representation theorem of any complete lattice as  sub $\bigcap$-semilattice of some powerset algebra (i.e.\ as the $\bigcap$-semilattice of the closed sets of a closure operator on that powerset algebra) and as  sub $\bigcup$-semilattice of some powerset algebra (i.e.\ as the $\bigcup$-semilattice of the open sets of an interior operator on that powerset algebra).  Each of these powerset algebras provide the semantics for  a different {\em type}, and their interaction with the original complete lattice is given as pairs of adjoint connectives, the composition of which yields the closure operator and the interior operator of the double representation. The proof-theoretic behaviour of the adjoint connectives is that of standard normal modal operators. In the multi-type environment, the interpretation of the sequents of the Hilbert-style axiomatization of lattice logic is then obtained via two translations, the soundness of which is justified by the double representations. The translated axiomatization of lattice logic is then derived in the multi-type proof calculus.  The metatheory of this calculus is smooth and encompassed in a general theory (cf.\ \cite{CoPa:Multi-type,Trends,CR}), so that one obtains soundness, completeness, conservativity and cut-elimination as easy corollaries of general facts.

\paragraph{Structure of the paper. } In Section \ref{sec:preliminaries}, we briefly report on a Hilbert-style presentation of  lattice logic and its algebraic semantics, and discuss the issue of a modular account of its axiomatic extensions and expansions. In Section \ref{sec:heterogeneous algebras for complete lattices}, we report on well known order-theoretic facts related with the representation of complete lattices, which help to introduce an equivalent multi-type semantic environment for lattice logic. In Section \ref{sec: multi-type language}, we introduce the multi-type language naturally associated with the semantic environment of the previous section. In Section \ref{MultiTypeDisplayCalculus}, we introduce the multi-type calculus D.LL for lattice logic which constitutes the core contribution of the present paper. In Section \ref{sec:properties}, we discuss the basic properties verified by D.LL, namely, soundness, completeness, cut-elimination, subformula property, and conservativity. In Section \ref{sec:distrib fails}, we prove syntactically that (the translation of) the distributivity axiom is not derivable in D.LL.

\section{Lattice logic and its single-type proof theory}
\label{sec:preliminaries}

\subsection{Hilbert-style presentation of lattice logic and its algebraic semantics}
\label{ssec:hilbert axioms}
 Formulas of  the language of lattice logic $\mathcal{L}$ over a set $\mathsf{AtProp}$ of atomic propositions are generated as follows: 

$$a ::= \,p \mid \topl \mid \botl \mid A \andl A \mid A \orl A.$$

The Hilbert-style presentation of lattice logic consists of the following axioms
\begin{align*}
				&A\vdash A, && \bot\vdash A, && A\vdash \top, & &\\ 
				&A\vdash A\vee B, && B\vdash A\vee B, && A\wedge B\vdash A, && A\wedge B\vdash B, &
			\end{align*}
and the following rules:
\begin{displaymath}
			\frac{A\vdash B\quad B\vdash C}{A\vdash C}
			\quad
			\frac{A\vdash B}{A[C/p]\vdash B[C/p]}
			\quad
			\frac{A\vdash B\quad A\vdash C}{A\vdash B\andl C}
			\quad
			\frac{A\vdash C\quad B\vdash C}{A\orl B\vdash C}
		\end{displaymath}
The algebraic semantics of lattice logic is given by the class of {\em bounded lattices} (cf.\ \cite{Birkhoff, Burris-Sankappanavar}), i.e.\ $(2, 2, 0,0)$-algebras $\mathbb{A} = (X, \andl, \orl, \topl, \botl)$ validating the following identities:
\begin{center}
\begin{tabular}{rlrl}
         & Commutative laws                          &          & Associative laws                                                            \\
cC. & $a \andl b = b \andl a$  & cA.   & $a \andl (b \andl c) = (a \andl b) \andl c$  \\
dC. & $a \orl b = b \orl a$       & dA.   & $a \orl (b \orl c) = (a \orl b) \orl c$              \\
        & Identity laws                                    &         & Absorption laws                                                                \\
cI.  & $a \andl \topl = a$         & cAb. & $a \andl (a \orl b) = a$                                  \\
dI.  & $a \orl \botl = a$            & dAb. & $a \orl (a \andl b) = a$                                  \\
\end{tabular}
\end{center}

A  bounded  lattice is {\em distributive} if it validates the following identities:

\begin{center}
\begin{tabular}{rl}
      & Distributivity laws                                                                 \\
cD. & $a \andl (b \orl c) = (a \andl b) \orl (a \orl c)$ \\
dD. & $a \orl (b \andl c) = (a \orl b) \andl (a \orl c)$ \\
\end{tabular}
\end{center}

 A bounded lattice is {\em residuated}  (cf.\ \cite{GJKO}) if the condition (cR) below holds, and is {\em dually residuated} if the condition (dR) holds. If a lattice is (dually) residuated then is distributive.  

\begin{center}
\begin{tabular}{rlclrlcl}
       & \mc{3}{l}{Residuation laws}                                    \\
cR. & $a \andl b \leq c$ &iff& $b \leq a \rarrl c$   \\ 
dR. & $a \leq b \orl c$    &iff& $b \drarrl a \leq c$ \\ 

\end{tabular}
\end{center}

\subsection{Towards a modular  proof theory for lattice logic}
In order to motivate the proposal of a calculus for lattice logic which we will introduce in Section \ref{MultiTypeDisplayCalculus}, we find it useful to start by discussing the properties of the following basic Gentzen-style sequent calculus for lattice logic (cf.\ e.g.\ \cite{TS00}):
\begin{itemize}
\item Identity and Cut rules
\end{itemize}

\begin{center}
\begin{tabular}{rl}
\AXC{$\phantom{p \fCenter p}$}
\RightLabel{\fns $Id$}
\UI$p \fCenter p$
\DisplayProof
 &
\AX$X \fCenter A$
\AX$A \fCenter Y$
\RightLabel{\fns $Cut$}
\BI$X \fCenter Y$
\DisplayProof
 \\
\end{tabular}
\end{center}

\begin{itemize}
\item Operational rules
\end{itemize}

\begin{center}
\begin{tabular}{rl}
\AXC{\phantom{$\botl \fCenter $}}
\LeftLabel{\fns $\botl$}
\UI$\botl \fCenter \TOPBOTL$
\DisplayProof
 &
\AX$X \fCenter \TOPBOTL$
\RightLabel{\fns $\botl$}
\UI$X \fCenter \botl$
\DisplayProof
 \\

 & \\

\AX$\TOPBOTL \fCenter X$
\LeftLabel{\fns $\topl$}
\UI$\topl \fCenter X$
\DisplayProof
 &
\AXC{\phantom{$ \fCenter \topl$}}
\RightLabel{\fns $\topl$}
\UI$\TOPBOTL \fCenter \topl$
\DisplayProof
 \\

 & \\

\AX$A_i \fCenter X$
\LeftLabel{\fns $\andl_i$}
\UI$A_1 \andl A_2 \fCenter X$
\DisplayProof
 &
\AX$X \fCenter A$
\AX$X \fCenter B$
\RightLabel{\fns $\andl$}
\BI$X \fCenter A \andl B$
\DisplayProof
\ \, \\

 & \\
\AX$A \fCenter X$
\AX$B \fCenter X$
\LeftLabel{\fns $\orl$}
\BI$A \orl B \fCenter X$
\DisplayProof
 &
\AX$X \fCenter A_i$
\RightLabel{\fns $\orl_i$}
\UI$X \fCenter A_1 \orl A_2$
\DisplayProof
 \\
 & \\
\mc{2}{c}{where $i \in\{1,2\}$}.
 \\
\end{tabular}
\end{center}
The calculus above, which we refer to as L0, is sound  w.r.t.\ the class of lattices, complete w.r.t.\ the Hilbert-style presentation of lattice logic, and verifies  cut-elimination. Hence, L0 is perfectly adequate as a proof calculus for lattice logic, when this logic is regarded in isolation. However, the main interest of lattice logic lays in it serving as base for a variety of  logics, which are either its {\em axiomatic extensions}  (e.g.~the logics of modular and distributive bounded lattices and their variations \cite{Huhn}), or its  proper {\em language-expansions}  (e.g.~the full Lambek calculus \cite{GJKO}, bilattice logic \cite{Be77}, orthologic \cite{Goldblatt},  linear logic \cite{Girard:TCS}). Hence, it is sensible to require of an adequate proof theory of lattice logic to be able to account in a modular way for these logics as well. The calculus L0 does not seem to be a good starting point for this purpose. Indeed,   axiomatic extensions of lattice logic can be supported by L0 by adding suitable axioms. For instance, modular and distributive lattice logic can be respectively captured by adding the following axioms to L0:
\[((C\andl B)\orl A)\andl B\vdash (C\andl B)\orl (A\andl B)\quad \quad A \andl (B \orl C) \vdash (A \andl B) \orl (A \orl C).\]
However, the cut elimination theorem needs to be proved for the resulting calculi from scratch. More in general, we lack uniform principles or proof strategies aimed at identifying axioms which can be added to L0 so that the resulting calculus still enjoys cut elimination. Another source of nonmodularity arises from the fact that L0 lacks structural rules. Indeed, the additive formulation of the introduction rules of L0 encodes the information which is stored in standard structural rules such as weakening, contraction, associativity, and exchange. Hence, one cannot use L0 as a base to capture logics aimed at `negotiating' these rules, such as the Lambek calculus \cite{La58} and other substructural logics \cite{GJKO}.
To remedy this, one can move to the following calculus, which we refer to as L1 and  which adopts the {\em visibility} principle\footnote{A sequent calculus verifies the {\em visibility} property if both the auxiliary formulas and the principal formula of each operational rule of the calculus occur in an {\em empty} context. Hence,  by design, L1 verifies the visibility property.} isolated by Sambin, Battilotti and Faggian in \cite{SambinBattilottiFaggian2000}  to formulate a general strategy for cut elimination. The visibility constraint generalizes Gentzen's seminal idea to capture intuitionistic logic with his calculus LJ by   restricting the shape of the sequents in his calculus LK for classical logic so as to admit at most one formula in succedent position \cite{Gentzen}. 
The calculus L1 has a  structural language, which  consists of one structural constant `$\TOPBOTL$' which is interpreted as $\topl$ (resp.\ $\botl$) when occurring in precedent (resp.\ succedent) position, and one binary connective `$\ANDORL$', which is interpreted  as conjunction in precedent position and disjunction in succedent position.

\begin{itemize}
\item Identity and Cut rules
\end{itemize}

\begin{center}
\begin{tabular}{ccc}
\AXC{$\phantom{p \fCenter p}$}
\RightLabel{\fns $Id$}
\UI$p \fCenter p$
\DisplayProof
 &
\AX$X \fCenter A$
\AX$(Y \fCenter Z)[A]^{pre}$
\LeftLabel{\fns L-Cut}
\BI$(X \fCenter Y)[Z/A]^{pre}$
\DisplayProof
 &
\AX$(X \fCenter Y)[A]^{succ}$
\AX$A \fCenter Z$
\RightLabel{\fns R-Cut}
\BI$(X \fCenter Y)[Z/A]^{suc}$
\DisplayProof
 \\
\end{tabular}
\end{center}

\begin{itemize}
\item Structural and operational  rules
\end{itemize}

\begin{center}
\begin{tabular}{@{}c@{}c@{}}


\mc{1}{c}{structural}
 &

\mc{1}{c}{\!\!\!\!\!\!\!\!\!\!
operational}
 \\

 & \\

\begin{tabular}{rl}
\AX$X \fCenter Y$
\doubleLine
\LeftLabel{\fns $\TOPBOTL$}
\UI$X \ANDORL \TOPBOTL \fCenter Y$
\DisplayProof
 &
\AX$X \fCenter Y$
\doubleLine
\RightLabel{\fns $\TOPBOTL$}
\UI$X \fCenter Y \ANDORL \TOPBOTL$
\DisplayProof
 \\

 & \\

\AX$X \ANDORL Y \fCenter Z$
\LeftLabel{\fns $E$}
\UI$Y \ANDORL X \fCenter Z$
\DisplayProof
 &
\AX$X \fCenter Y \ANDORL Z$
\RightLabel{\fns $E$}
\UI$X \fCenter Z \ANDORL Y$
\DisplayProof \\

 & \\

\AX$(X \ANDORL Y) \ANDORL Z \fCenter V$
\doubleLine
\LeftLabel{\fns $A$}
\UI$X \ANDORL (Y \ANDORL Z) \fCenter V$
\DisplayProof
 &
\AX$X \fCenter (Y \ANDORL Z) \ANDORL V$
\RightLabel{\fns $A$}
\doubleLine
\UI$X \fCenter Y \ANDORL (Z \ANDORL V)$
\DisplayProof
 \\

 & \\

\AX$X \fCenter Y$
\LeftLabel{\fns $W$}
\UI$X \ANDORL Z \fCenter Y$
\DisplayProof
 &
\AX$X \fCenter Y$
\RightLabel{\fns $W$}
\UI$X \fCenter Y \ANDORL Z$
\DisplayProof
 \\

 & \\

\AX$X \ANDORL X \fCenter Y$
\LeftLabel{\fns $C$}
\UI$X \fCenter Y$
\DisplayProof
 &
\AX$X \fCenter Y \ANDORL Y$
\RightLabel{\fns $C$}
\UI$X \fCenter Y$
\DisplayProof
 \\




\end{tabular}

 &
\!\!\!\!\!\!\!\!\!\!
\begin{tabular}{rl}
\AXC{\phantom{$\botl \fCenter $}}
\LeftLabel{\fns $\botl$}
\UI$\botl \fCenter \TOPBOTL$
\DisplayProof
 &
\AX$X \fCenter \TOPBOTL$
\RightLabel{\fns $\botl$}
\UI$X \fCenter \botl$
\DisplayProof
 \\

 & \\

\AX$\TOPBOTL \fCenter X$
\LeftLabel{\fns $\topl$}
\UI$\topl \fCenter X$
\DisplayProof
 &
\AXC{\phantom{$ \fCenter \topl$}}
\RightLabel{\fns $\topl$}
\UI$\TOPBOTL \fCenter \topl$
\DisplayProof
 \\

 & \\

\AX$A \ANDORL B \fCenter X$
\LeftLabel{\fns $\andl$}
\UI$A \andl B \fCenter X$
\DisplayProof
 &
\AX$X \fCenter A$
\AX$Y \fCenter B$
\RightLabel{\fns $\andl$}
\BI$X \ANDORL Y \fCenter A \andl B$
\DisplayProof
\ \, \\

 & \\

\AX$A \fCenter X$
\AX$B \fCenter Y$
\LeftLabel{\fns $\orl$}
\BI$A \orl B \fCenter X \ANDORL Y$
\DisplayProof
 &
\AX$X \fCenter A \ANDORL B$
\RightLabel{\fns $\orl$}
\UI$X \fCenter A \orl B$
\DisplayProof
 \\
\end{tabular}

\end{tabular}
\end{center}
Unlike the operational rules for L0, the operational rules for L1 are formulated in
{\em multiplicative} form,\footnote{The multiplicative form of the introduction rules is the most important aspect in which L1 departs from the calculus of \cite{SambinBattilottiFaggian2000}. Indeed, the introduction rules for conjunction and disjunction in \cite{SambinBattilottiFaggian2000} are additive.}
 which is more general than the additive.  
The more general formulation of the introduction rules implies that the structural rules of weakening, exchange, associativity, and contraction are not anymore subsumed by the introduction rules. 

The visibility of L1 blocks the derivation of the distributivity axiom. Hence, to be able to derive distributivity, one option is to relax the visibility constraint both in precedent  and in succedent position. This solution is not entirely satisfactory, and suffers from the same lack of modularity which prevents    Gentzen's move from LJ to LK to capture intermediate logics. Specifically, relaxing visibility captures the logics of Sambin's cube, but many other logics are left out. Moreover, without visibility, we do not have a uniform strategy for cut elimination. 

To conclude,  a proof theory for axiomatic extensions and expansions of general lattice logic is comparably not as modular as that of the axiomatic extensions and expansions of the logic of {\em distributive} lattices, which can rely on the theory of proper display calculi \cite{wansing, GMPTZ}. The idea guiding the approach of the present paper, which we will elaborate upon in the next sections, is that, rather than trying to work our way up starting from a calculus for lattice logic, we will  obtain a  calculus for lattice logic from the standard proper display calculus for the logic of distributive lattices, by endowing it with a suitable mechanism to block the derivation of distributivity.
\section{Multi-type semantic environment for lattice logic}
\label{sec:heterogeneous algebras for complete lattices}
In the present section, we introduce a class of {\em heterogeneous algebras} \cite{BL} which equivalently encodes complete lattices, and which will be useful to motivate the design of the calculus for lattice logic from a semantic viewpoint, as well as to establish its properties. This presentation  takes its move from very well known facts in the representation theory of complete lattices, which can be found e.g.~in \cite{DaPr, Birkhoff}, formulated---however---in terms of {\em covariant} (rather than contravariant) adjunction. For every partial order $\mathbb{Q} = (Q, \leq)$, we let $\mathbb{Q}^{op}: = (Q, \leq^{op})$, where $\leq^{op}$ denotes the converse ordering.  If $\mathbb{Q} = (Q, \wedge, \vee, \bot, \top)$ is a lattice, we let $\mathbb{Q}^{op}: = (Q, \wedge^{op}, \vee^{op}, \bot^{op}, \top^{op})$ denote the lattice induced by $\leq^{op}$. Moreover, for any $b\in Q$, we let $b{\uparrow}: = \{c\mid c\in Q\mbox{ and } b\leq c\}$ and $b{\downarrow}: = \{a\mid a\in Q\mbox{ and } a\leq b\}$.

\medskip

A {\em polarity} is a structure $\mathbb{P} = (X, Y, R)$ such that $X$ and $Y$ are sets and $R\subseteq X\times Y$. Every polarity induces a pair of maps  $\rho: \mathcal{P}(Y)^{op}\to \mathcal{P}(X)$, $\lambda: \mathcal{P}(X)\to \mathcal{P}(Y)^{op}$, respectively defined by $Y'\mapsto \{x\in X\mid \forall y(y\in Y'\rightarrow xR y) \}$ and $X'\mapsto \{y\in Y\mid \forall x(x\in X'\rightarrow xR y) \}$. It is well known (cf.~\cite{DaPr}) and easy to verify that these maps form an {\em adjunction pair}, that is, for any $X'\subseteq X$ and $Y'\subseteq Y$,
\[\lambda(X')\subseteq^{op} Y'\quad \mbox{ iff }\quad X'\subseteq \rho (Y'). \]
The map $\lambda$ is the left adjoint, and $\rho$ is the right adjoint of the pair. By general order-theoretic facts, this implies that $\lambda$ preserves arbitrary joins and $\rho$ arbitrary meets: that is, for any $S\subseteq \mathcal{P}(X)$ and any $T\subseteq \mathcal{P}(Y)$,
\begin{equation}
\label{eq:lambda join pres rho met pres}
 \lambda (\bigcup S) = \bigcup^{op}_{s\in S}\lambda(s)\quad\mbox{ and }\quad \rho (\bigcap^{op} T) = \bigcap_{t\in T} \rho (t).
 \end{equation}
 Other well known facts about adjoint pairs are that $\rho\lambda: \mathcal{P}(X)\to \mathcal{P}(X)$ is a closure operator and $\lambda\rho: \mathcal{P}(Y)^{op}\to \mathcal{P}(Y)^{op}$  an interior operator (cf.~\cite{DaPr}). Moreover,  $\lambda\rho\lambda = \lambda$, and  $\rho\lambda\rho = \rho$ (cf.~\cite{DaPr}). That is, $\lambda\rho$ restricted to  $\mathsf{Range}(\lambda)$ is the identity map, and likewise, $\rho\lambda$ restricted to $\mathsf{Range}(\rho)$ is the identity map. Hence, $\mathsf{Range}(\rho) = \mathsf{Range}(\rho\lambda)$,  $\mathsf{Range}(\lambda) = \mathsf{Range}(\lambda\rho)$ and \[\mathcal{P}(X)\supseteq \mathsf{Range}(\rho)\cong \mathsf{Range}(\lambda)\subseteq \mathcal{P}(X)^{op}.\]
 Furthermore,  $\rho\lambda$ being a closure operator on $\mathcal{P}(X)$ implies that $\mathsf{Range}(\rho)= \mathsf{Range}(\rho\lambda)$ is a complete sub $\bigcap$-semilattice of $\mathcal{P}(X)$ (cf.~\cite{DaPr}), and hence $\mathbb{L} = \mathsf{Range}(\rho)$ is endowed with a structure of complete lattice, by setting  for every $S\subseteq \mathbb{L}$,
\begin{equation}
\label{eq:def left lattice ops}
\bigwedge_{\mathbb{L}} S := \bigcap S\quad \mbox{ and }\quad \bigvee_{\mathbb{L}} S: = \rho\lambda(\bigcup S)
\end{equation}
 Likewise, $\lambda\rho$ being an interior operator on $\mathcal{P}(Y)^{op}$ implies that  $\mathsf{Range}(\lambda)$ is a complete sub $\bigcup$-semilattice of $\mathcal{P}(Y)^{op}$, and hence $\mathbb{L} = \mathsf{Range}(\lambda)$ is endowed with a structure of complete lattice, by setting
 \begin{equation}
\label{eq:def right lattice ops}\bigvee_{\mathbb{L}} T := \bigcup^{op} T\quad \mbox{ and }\quad \bigwedge_{\mathbb{L}} T: = \lambda\rho(\bigcap^{op} T)
\end{equation}
for every $T\subseteq \mathbb{L}$. Finally, for any $S\subseteq \mathsf{Range}(\rho)$,
\begin{center}
\begin{tabular}{r c l l}
$\lambda (\bigvee S)$ & $ = $& $ \lambda (\rho\lambda (\bigcup S))$ & \eqref{eq:def left lattice ops}\\
& $ = $& $\lambda (\bigcup S)$& $\lambda\rho \lambda=\lambda$\\
& $ = $& $ \bigcup^{op}_{s\in S} \lambda (s)$ & \eqref{eq:lambda join pres rho met pres} \\
& $ = $& $\bigvee_{s\in S}\lambda (s),$ & \eqref{eq:def right lattice ops}\\
 \end{tabular}
\end{center}
and
\begin{center}
\begin{tabular}{r c l l}
$\bigwedge_{s\in S}\lambda (s)$ & $ = $& $\lambda\rho(\bigcap^{op}_{s\in S}\lambda(s))$ & \eqref{eq:def right lattice ops}\\
& $ = $& $\lambda(\bigcap_{s\in S}\rho\lambda(s))$ & \eqref{eq:lambda join pres rho met pres}\\
& $ = $& $\lambda(\bigcap S)$ & $S\subseteq \mathsf{Range}(\rho)$ and $\rho\lambda\rho = \rho$\\
& $ = $& $\lambda (\bigwedge S)$, &\eqref{eq:def left lattice ops}\\
\end{tabular}
\end{center}
which shows that the restriction of $\lambda$ to $\mathsf{Range}(\rho)$ is a complete lattice homomorphism. Likewise, one can show that the restriction of $\rho$ to $\mathsf{Range}(\lambda)$ is a complete lattice homomorphism, which completes the proof that the bijection \[\mathcal{P}(X)\supseteq \mathsf{Range}(\rho)\cong \mathsf{Range}(\lambda)\subseteq \mathcal{P}(X)^{op}\] is in fact an isomorphism of complete lattices, and hence the abuse of notation is justified which we made by denoting both the lattice $\mathsf{Range}(\rho)$ and the lattice $\mathsf{Range}(\lambda)$ by $\mathbb{L}$.

Conversely, for every complete lattice $\mathbb{L}$, consider the polarity $\mathbb{P}_{\mathbb{L}}: =(L, L, \leq)$ where $L$ is the universe of  $\mathbb{L}$ and $\leq$ is the lattice order. Then the maps $\lambda: \mathcal{P}(L)\to \mathcal{P}(L)^{op}$ and $\rho: \mathcal{P}(L)^{op}\to \mathcal{P}(L)$ are respectively defined by the assignments $S\mapsto\{a\in L\mid \forall b(b\in S\rightarrow b\leq  a) \} = (\bigvee S){\uparrow}$ and $T\mapsto \{a\in L\mid \forall b(b\in T\rightarrow a\leq b) \} = (\bigwedge T){\downarrow}$ for all $S, T\subseteq L$. Since $\bigwedge ((\bigvee S){\uparrow}) = \bigvee S$ and $\bigvee((\bigwedge T){\downarrow}) = \bigwedge T$, the closure operator $\rho\lambda: \mathcal{P}(L)\to \mathcal{P}(L)$  and the interior operator $\lambda\rho: \mathcal{P}(L)^{op}\to \mathcal{P}(L)^{op}$ are respectively defined by
\begin{equation}
\label{eq:def of closure and interior ops}
S\mapsto (\bigvee S){\downarrow}\quad\mbox{ and }\quad  T\mapsto (\bigwedge T){\uparrow}.
 \end{equation} 
The lattice $\mathbb{L}$ can be mapped injectively  both into $\mathsf{Range}(\rho) = \mathsf{Range}(\rho\lambda)$ and into $\mathsf{Range}(\lambda) = \mathsf{Range}(\lambda\rho)$ by the assignments $a\mapsto  a{\downarrow}$ and $a\mapsto a{\uparrow}$ respectively. Moreover, since $\mathbb{L}$ is complete, the maps  defined by these assignments are also {\em onto} $\mathsf{Range}(\rho\lambda)$ and $\mathsf{Range}(\lambda\rho)$. Finally, for any $S\subseteq \mathbb{L}$,
\begin{center}
\begin{tabular}{r c l l}
$\bigwedge_{\mathsf{Range}(\rho)}\{a{\downarrow}\mid a\in S\}$ & $ = $& $\bigcap\{a{\downarrow}\mid a\in S\}$ & \eqref{eq:def left lattice ops}\\
& $ = $& $(\bigwedge S){\downarrow}$ & \\
&& \\
$\bigvee_{\mathsf{Range}(\rho)}\{a{\downarrow}\mid a\in S\}$ & $ = $& $\rho\lambda(\bigcup\{a{\downarrow}\mid a\in S\})$ & \eqref{eq:def left lattice ops}\\
& $ = $& $(\bigvee \bigcup\{a{\downarrow}\mid a\in S\}){\downarrow}$ & \eqref{eq:def of closure and interior ops}\\
& $ = $& $(\bigvee S){\downarrow}$, & \\
\end{tabular}
\end{center}
which completes the verification that the map $\mathbb{L}\to \mathsf{Range}(\rho)$ defined by the assignment $a\mapsto  a{\downarrow}$ is  a complete lattice isomorphism. Similarly, one verifies that the map $\mathbb{L}\to \mathsf{Range}(\lambda)$ defined by the assignment $a\mapsto  a{\uparrow}$ is a complete lattice isomorphism. The discussion so far can be summarized by the following
\begin{proposition}
Any complete lattice $\mathbb{L}$ can be identified both with the lattice of closed sets of some closure operator $c: \mathbb{D}\to \mathbb{D}$ on a complete and completely distributive lattice $\mathbb{D} = (D, \cap, \cup, \ptopp, \varnothing)$, and with the lattice of open sets of some interior operator $i: \mathbb{E}\to \mathbb{E}$ on a complete and completely distributive lattice $\mathbb{E} = (E, \sqcap, \sqcup, \Im, \emptyset)$.
\end{proposition}
Hence, in what follows, $\mathbb{L}$ will be identified both with $\mathsf{Range}(c)$ endowed with its structure of complete lattice defined as in \eqref{eq:def left lattice ops} (replacing $\rho\lambda$ by $c$), and with $\mathsf{Range}(i)$ endowed with its  structure of complete lattice defined as in \eqref{eq:def right lattice ops} (replacing $\lambda\rho$ by $i$). Taking these identifications into account, general order-theoretic facts  (cf.~\cite[Chapter 7]{DaPr}) imply that  $c = e_{\ell}\circ \gamma$, where $\gamma: \mathbb{D}\twoheadrightarrow \mathbb{L}$ is defined by $\alpha\mapsto c(\alpha)$ and  $e_\ell: \mathbb{L}\hookrightarrow \mathbb{D}$ is the natural embedding, and moreover, these maps form an adjunction pair as follows: for any $a\in \mathbb{L}$ and any $\alpha\in \mathbb{D}$,
\[\gamma(\alpha)\leq a\quad \mbox{ iff } \quad\alpha\leq e_{\ell}(a),\]
with the additional property that $\gamma\circ e_\ell = Id_{\mathbb{L}}$. Likewise, $i = e_{r}\circ \iota$, where $\iota: \mathbb{E}\twoheadrightarrow \mathbb{L}$ is defined by $\xi\mapsto i(\xi)$ and  $e_r: \mathbb{L}\hookrightarrow \mathbb{E}$ is the natural embedding, and moreover, these maps form an adjunction pair as follows: for any $a\in \mathbb{L}$ and any $\xi\in \mathbb{E}$,
\[ e_{r}(a)\leq\xi \quad \mbox{ iff } \quad a\leq \iota(\xi),\]
with the additional property that $\iota\circ e_r = Id_{\mathbb{L}}$.
\begin{center}
\begin{tikzpicture}[node/.style={circle, draw, fill=black}, scale=1]

\node (D) at (-1.5,-1.5) {$\mathbb{D}$};
\node (L) at (1.5,-1.5) {$\mathbb{L}$};
\node (E) at (4.5,-1.5) {$\mathbb{E}$};
\node (adj) at (3,-1.1) {{\rotatebox[origin=c]{270}{$\vdash$}}};
\node (adj) at (0,-1.1) {{\rotatebox[origin=c]{90}{$\vdash$}}};
\draw [left hook->] (L)  to node[below] {$e_{\ell}$}  (D);
\draw [right hook->] (L)  to node[below] {$e_{r}$}  (E);

\draw [->>] (E) to [out=135,in=45, looseness=1]   node[above] {$\iota$}  (L);
\draw [<<-] (L) to [out=135,in=45, looseness=1]   node[above] {$\gamma$}  (D);
\end{tikzpicture}
\end{center}
Summing up, any complete lattice $\mathbb{L}$ can be associated with an heterogeneous algebra $(\mathbb{L}, \mathbb{D}, \mathbb{E}, e_\ell, \gamma, e_r, \iota)$ such that
\begin{itemize}
\item[H1.] $\mathbb{L} = (L, \leq)$ is a bounded poset;\footnote{We overload the symbol $\mathbb{L}$ and use it both to denote the complete lattice and its underlying poset.}
\item[H2.] $\mathbb{D}$ and $\mathbb{E}$ are complete and completely distributive lattices;
\item[H3.] $\gamma: \mathbb{D}\to \mathbb{L}$ and $e_\ell: \mathbb{L}\to \mathbb{D}$ are such that $\gamma\dashv e_\ell$ and $\gamma \circ e_\ell = Id_{\mathbb{L}}$;
\item[H4.] $\iota: \mathbb{E}\to \mathbb{L}$ and $e_r: \mathbb{L}\to \mathbb{E}$ are such that $ e_r \dashv\iota $ and $\iota \circ e_r = Id_{\mathbb{L}}$.
\end{itemize}
Conversely, for any such an heterogeneous algebra, the poset $\mathbb{L}$ can be endowed with the structure of a complete lattice inherited by being order-isomorphic both to the poset of closed sets of the closure operator $c: = \gamma\circ e_{\ell}$ on $\mathbb{D}$ and to the poset of open sets of the interior operator $i: = \iota\circ e_{r}$ on $\mathbb{E}$. Finally, no algebraic information is lost when presenting a complete lattice $\mathbb{L}$ as its associated heterogeneous algebra. Indeed,
the identification of $\mathbb{L}$ with $\mathsf{Range}(c)$, endowed with the structure of complete lattice defined as in \eqref{eq:def left lattice ops}, implies that for all $a, b\in \mathbb{L}$,
\[a\vee b =  \gamma(e_\ell(a)\cup e_\ell (b)). \]
As discussed above,  $e_\ell$ being a right adjoint and $\gamma$ a left adjoint imply that $e_\ell$ is completely meet-preserving and $\gamma$ completely join-preserving. Therefore, $e_\ell(\top) = \ptopp$ and $\bot  = \gamma(\varnothing)$. Moreover,  $\gamma$ being both surjective and order-preserving implies that $\top = \gamma(\ptopp)$. Furthermore,  for all $a, b\in \mathbb{L}$,
\[a\wedge b = \gamma\circ e_\ell(a\wedge b) = \gamma(e_\ell(a)\cap e_\ell (b)). \]
Thus, the whole algebraic structure of $\mathbb{L}$ can be captured in terms of the algebraic structure of $\mathbb{D}$ and the adjoint maps $\gamma$ and $e_\ell$ as follows: for all $a, b\in \mathbb{L}$,
\begin{equation}
\label{eq:soundness left-translation}
\bot  = \gamma(\varnothing)  \quad \top = \gamma(\ptopp) \quad a\vee b =  \gamma(e_\ell(a)\cup e_\ell (b)) \quad a\wedge b =  \gamma(e_\ell(a)\cap e_\ell (b)). \end{equation}
Reasoning analogously, one can also capture the algebraic structure of $\mathbb{L}$ in terms of the algebraic structure of $\mathbb{E}$ and the adjoint maps $\iota$ and $e_r$ as follows: for all $a, b\in \mathbb{L}$,
\begin{equation}
\label{eq:soundness right-translation}
\top = \iota(\Im) \quad \bot = \iota(\emptyset)  \quad a\wedge b =  \iota(e_r(a)\sqcap e_r (b)) \quad a\vee b = \iota(e_r(a)\sqcup e_r (b)).
\end{equation}

\section{Multi-type Hilbert-style presentation for lattice logic}\label{sec: multi-type language}
In Section \ref{sec:heterogeneous algebras for complete lattices},  heterogeneous algebras  have been introduced  and shown to be  equivalent presentations of  complete lattices. The toggle between these mathematical structures is reflected in the toggle between the logical languages which are naturally interpreted in the two types of structures.  Indeed, the heterogeneous algebras of Section \ref{sec:heterogeneous algebras for complete lattices} provide a natural interpretation for the following multi-type language $\mathcal{L}_{\mathrm{MT}}$ over a set $\mathsf{AtProp}$ of $\mathsf{Lattice}$-type atomic propositions:
\begin{align*}
\mathsf{Left}\ni \alpha ::=&\, e_{\ell}(A)    \mid \ptopp \mid \varnothing \mid \alpha \cup \alpha \mid  \alpha \cap \alpha \\
\mathsf{Right}\ni \xi ::=&\,  e_{r}(A)        \mid \Im \mid \emptyset \mid \xi \sqcup \xi \mid \xi \sqcap \xi  \\
\mathsf{Lattice}\ni  A ::= & \,p \mid \, \gamma(\alpha) \mid \iota(\xi) \mid \top \mid \bot
\end{align*}
where $p\in \mathsf{AtProp}$.
The interpretation of $\mathcal{L}_{\mathrm{MT}}$-terms  into heterogeneous algebras  is defined as the straightforward generalization of the interpretation of propositional languages in algebras of compatible signature.
At the end of the previous section, we observed  that the algebraic structure of the complete lattice $\mathbb{L}$ can be captured in terms of the algebraic structure of its associated heterogeneous algebra. This observation serves as a base for the definition of the translations $(\cdot)^\ell, (\cdot)^r: \mathcal{L}\to \mathcal{L}_{\mathrm{MT}}$
between the original language $\mathcal{L}$ of lattice logic  and  $\mathcal{L}_{\mathrm{MT}}$:
\begin{center}
\begin{tabular}{r c l c r c l}
$p^\ell$ &$ = $& $\gamma e_\ell(p)$ & $\quad\quad$ &$p^r$ &$ = $& $\iota e_r(p)\mneg$\\
$\top^\ell$ &$ = $& $\gamma e_\ell(\top)$ && $\top^r$ &$ = $& $\iota e_r(\top)$\\
$\bot^\ell$ &$ = $& $\gamma e_\ell(\bot)$ && $\bot^r$ &$ = $& $\iota e_r(\bot)$\\
$(A\andl B)^\ell$ &$ = $& $\gamma (e_\ell(A^\ell) \cap e_\ell(B^\ell))$ && $(A\andl B)^r$ &$ = $& $\iota (e_r(A^r) \sqcap e_r(B^r))$\\
$(A\orl B)^\ell$ &$ = $& $\gamma (e_\ell(A^\ell) \cup e_\ell(B^\ell))$ && $(A\orl B)^r$ &$ = $& $\iota (e_r(A^r) \sqcup e_r(B^r))$\\
\end{tabular}
\end{center}
For  every complete lattice $\mathbb{L}$, let $\mathbb{L}^*$ denote its associated heterogeneous algebra as defined in Section \ref{sec:heterogeneous algebras for complete lattices}. The proof of the following proposition relies on the observations made at the end of Section \ref{sec:heterogeneous algebras for complete lattices}. 
\begin{proposition}
\label{prop:consequence preserved and reflected}
For all $\mathcal{L}$-formulas $A$ and $B$ and every complete lattice $\mathbb{L}$,
\[\mathbb{L}\models A\leq B \quad \mbox{ iff }\quad \mathbb{L}^*\models A^\ell\leq B^r.\]
\end{proposition}
\section{Proper display calculus for lattice logic}
\label{MultiTypeDisplayCalculus}
In the present section, we introduce the proper multi-type display calculus D.LL for lattice logic.
\subsection{Language}\label{ssec:language of DLL}
The language of D.LL  includes the types $\mathsf{Lattice}$, $\mathsf{Left}$, and $\mathsf{Right}$, sometimes abbreviated as $\mathsf{L}$, $\mathsf{P}$, and $\mathsf{P}^{\mathrm{op}}$ respectively.

\begin{center}
\begin{tabular}{l}
$\mathsf{^{\phantom{\mathrm{op}}}L} \left\{\begin{array}{l}
A ::= \,p \mid \pbdiabp \alpha \mid \bboxb \xi \\
 \\
X ::= \, p \mid \TOPBOTL \mid \PBCIRCBP \Gamma \mid \PBCIRCBP^{\mathrm{op}} \Pi \\
\end{array} \right.$
 \\

 \\

$\mathsf{^{\phantom{\mathrm{op}}}P} \left\{\begin{array}{l}
\alpha     ::= \, \pwbox A \\
 \\
\Gamma ::= \, \PWCIRC X \mid \PTOPBOT \mid \Gamma \PANDOR \Gamma \mid \Gamma\PARR \Gamma \\
\end{array} \right.$
 \\

 \\

$\mathsf{P^{\mathrm{op}}} \left\{\begin{array}{l}
\xi     ::= \, \diawp A \\
 \\
\Pi ::= \, \CIRCWP X \mid \TOPBOTP \mid \Pi \ANDORP \Pi \mid \Pi \ARRP \Pi \\
\end{array} \right.$

 \\
\end{tabular}
\end{center}

Our notational conventions assign different variables to different types, and hence allow us to drop the subscripts $^{\mathrm{op}}$, given that the parsing of expressions such as $\PBCIRCBP \Gamma $ and $\PBCIRCBP \Pi$ is unambiguous.
\begin{itemize}
\item Structural and operational pure $\mathsf{L}$-type connectives:\footnote{We follow the notational conventions introduced in \cite{LORI}: Each structural connective in the upper row of the synoptic tables is interpreted as  the logical connective in the left (resp.\ right) slot  in the lower row when occurring in precedent (resp.\ succedent) position.}
\end{itemize}

\begin{center}
\begin{tabular}{|c|c|}
\hline
\mc{2}{|c|}{\fns $\mathsf{L}$ connectives} \\
\hline
\mc{2}{|c|}{$\TOPBOTL$}                           \\
\hline
$\ \ \topl\ \ $ & $\botl$                                    \\
\hline
\end{tabular}
\end{center}

\begin{itemize}
\item Structural and operational pure $\mathsf{P}$-type and $\mathsf{P^\mathrm{op}}$-type connectives:
\end{itemize}

\begin{center}
\begin{tabular}{|c|c|c|c|c|c|}
\hline
\mc{6}{|c|}{\fns $\mathsf{P}$ connectives}                                                   \\
\hline
\mc{2}{|c|}{$\PTOPBOT$} & \mc{2}{c|}{$\PANDOR$} & \mc{2}{c|}{$\PARR$}   \\
\hline
$(\ptopp)$ & $(\pbot)$        & $\pand$ & $\por$         & $(\pdrarr)$ & $(\prarr)$ \\
\hline
\end{tabular}
\end{center}


\begin{center}
\begin{tabular}{|c|c|c|c|c|c|}
\hline
\mc{6}{|c|}{\fns $\mathsf{P^\mathrm{op}}$ connectives}                                                   \\
\hline
\mc{2}{|c|}{$\,\TOPBOTP$} & \mc{2}{c|}{$\ \ \ \,\ANDORP$} & \mc{2}{c|}{$\ \ \ \,\ARRP$}   \\
\hline
$(\topp)$ & $(\botp)$        & $\andp$ & $\orp$         & $(\drarrp)$ & $(\rarrp)$ \\
\hline
\end{tabular}
\end{center}


\begin{itemize}
\item Structural and operational multi-type connectives: 
\end{itemize}

\begin{center}
\begin{tabular}{|c|c|c|c|c|c|c|c|}
\hline
\mc{2}{|c|}{\fns $\mathsf{L} \to \mathsf{P}$} & \mc{2}{c|}{\fns $\mathsf{L} \to \mathsf{P^\mathrm{op}}$} & \mc{2}{c|}{\fns $\mathsf{P} \to \mathsf{L}$} & \mc{2}{c|}{\fns $\mathsf{P^\mathrm{op}} \to \mathsf{L}$} \\
\hline
\mc{2}{|c|}{$\PWCIRC$}               & \mc{2}{c|}{$\CIRCWP$}               & \mc{2}{c|}{$\PBCIRC$} & \mc{2}{c|}{$\PBCIRC$}           \\
\hline
$\phantom{\pwbox}$ & $\pwbox$ & $\diawp$ & $\phantom{\diawp}$ & $\pbdia$ & $\phantom{\pbdia}$ & $\phantom{\boxbp}$ & $\boxbp$ \\
\hline
\end{tabular}
\end{center}

The connectives $\pwbox$,  $\diawp$, $\pbdia$ and $\boxbp$ are interpreted  in heterogeneous algebras  as the maps $e_\ell$, $e_r$, $\gamma$, and  $\iota$,  respectively.



\subsection{Rules}
In what follows, structures of type $\mathsf{L}$ are denoted by the variables $X, Y, Z$, and $W$; structures of type $\mathsf{P}$ are denoted by the variables $\Gamma, \Delta, \Theta$, and $\Lambda$; structures of type $\mathsf{P}^{\mathrm{op}}$ are denoted by the variables $\Pi, \Sigma, \Psi$, and $\Omega$.
Given the semantic environment introduced in Section \ref{sec:heterogeneous algebras for complete lattices}, it will come as no surprise that there is a perfect match between the  pure $\mathsf{P}$-type rules  and the pure $\mathsf{P^\mathrm{op}}$-type rules. In order to achieve a more compact presentation of the calculus, in what follows we will also reserve the variables $S, T, U$, and $V$ to denote either $\mathsf{P}$-type structures or  $\mathsf{P}^{\mathrm{op}}$-type structures, and $s, t, u$ and $ v$ to denote operational terms of either $\mathsf{P}$-type  or  $\mathsf{P}^{\mathrm{op}}$-type, with the proviso that they should be interpreted in the {\em same} type in the same  pure type-rule.
\begin{itemize}
\item Multi-type display rules
\end{itemize}

\begin{center}
\begin{tabular}{ccc}
\AX$\Gamma \fCenter \PWCIRCWP  X $
\RightLabel{\fns $D_{P\textrm{-}L}$}
\doubleLine
\UI$ \PBCIRCBP \Gamma \fCenter X $
\DisplayProof
 &
\AX$\PWCIRCWP  X  \fCenter \Pi$
\RightLabel{\fns $D_{P\textrm{-}L}$}
\doubleLine
\UI$ X \fCenter \PBCIRCBP \Pi$
\DisplayProof
 \\
\end{tabular}
\end{center}

\begin{itemize}
\item Pure $\mathsf{P}$-type and  $\mathsf{P}^{\mathrm{op}}$-type display rules
\end{itemize}

\begin{center}
\begin{tabular}{cc}
\AX$S \PANDORP T \fCenter U$
\LeftLabel{\fns $D_P$}
\doubleLine
\UI$T \fCenter S \PARRP U$
\DisplayProof
 &
\AX$S \fCenter T \PANDORP U$
\RightLabel{\fns $D_P$}
\doubleLine
\UI$T \PARRP S \fCenter U$
\DisplayProof
 \\
\end{tabular}
\end{center}

\begin{itemize}
\item Pure $\mathsf{P}$-type and  $\mathsf{P}^{\mathrm{op}}$-type rules
\end{itemize}

\begin{center}
\begin{tabular}{rl}

\mc{2}{c}{structural rules} \\

 & \\

\mc{2}{c}{
\AX$S \fCenter s$
\AX$s \fCenter T$
\RightLabel{\fns $Cut$}
\BI$S \fCenter T$
\DisplayProof
 }
 \\

 & \\

\AX$S \fCenter T$
\doubleLine
\LeftLabel{\fns $\PTOPBOTP$}
\UI$S \PANDORP \PTOPBOTP \fCenter T$
\DisplayProof
 &
\AX$S \fCenter T$
\doubleLine
\RightLabel{\fns $\PTOPBOTP$}
\UI$S \fCenter T \PANDORP \PTOPBOTP$
\DisplayProof
 \\

 & \\

\AX$S \PANDORP T \fCenter U$
\LeftLabel{\fns $E$}
\UI$T \PANDORP S \fCenter U$
\DisplayProof
 &
\AX$S \fCenter T \PANDORP U$
\RightLabel{\fns $E$}
\UI$S \fCenter U \PANDORP T$
\DisplayProof
 \\

 & \\

\AX$(S \PANDORP T) \PANDORP U \fCenter V$
\doubleLine
\LeftLabel{\fns $A$}
\UI$S \PANDORP (T \PANDORP U) \fCenter V$
\DisplayProof
 &
\AX$S \fCenter (T \PANDORP U) \PANDORP V$
\RightLabel{\fns $A$}
\doubleLine
\UI$S \fCenter T \PANDORP (U \PANDORP V)$
\DisplayProof
 \\

 & \\

\AX$S \fCenter T$
\LeftLabel{\fns $W$}
\UI$S \PANDORP U \fCenter T$
\DisplayProof
 &
\AX$S \fCenter T$
\RightLabel{\fns $W$}
\UI$S \fCenter T \PANDORP U$
\DisplayProof
 \\

 & \\

\AX$S \PANDORP S \fCenter T$
\LeftLabel{\fns $C$}
\UI$S \fCenter T$
\DisplayProof
 & 
\AX$S \fCenter T \PANDORP T$
\RightLabel{\fns $C$}
\UI$S \fCenter T$
\DisplayProof
 \\


 & \\

\mc{2}{c}{operational rules} \\

 & \\





\AX$s \PANDORP t \fCenter S$
\LeftLabel{\fns $\pandp$}
\UI$s \pandp t \fCenter S$
\DisplayProof
 &
\AX$S \fCenter s$
\AX$T \fCenter t$
\RightLabel{\fns $\pandp$}
\BI$S \PANDORP T \fCenter s \pandp t$
\DisplayProof
 \\

 & \\

\AX$s \fCenter S$
\AX$t \fCenter T$
\LeftLabel{\fns $\porp$}
\BI$s \porp t \fCenter S \PANDORP T$
\DisplayProof
 & 
\AX$S \fCenter s \PANDORP t$
\RightLabel{\fns $\porp$}
\UI$S \fCenter s \porp t$
\DisplayProof
 \\

\end{tabular}
\end{center}


\begin{itemize}
\item Pure $\mathsf{L}$-type rules
\end{itemize}

\begin{center}
\begin{tabular}{rlcrl}
\mc{2}{c}{structural rules} & \ \ \ \ \ & \mc{2}{c}{operational rules} \\
 & & & & \\
\AXC{\phantom{$p \fCenter p$}}
\LeftLabel{\fns $Id$}
\UI$p \fCenter p$
\DisplayProof
 &
\AX$X \fCenter A$
\AX$A \fCenter Y$
\RightLabel{\fns $Cut$}
\BI$X \fCenter Y$
\DisplayProof
 & &
\AX$\ATOPBOT \fCenter X$
\LeftLabel{\fns $\topl$}
\UI$\topl \fCenter X$
\DisplayProof
 &
\AXC{\phantom{$ \fCenter \top$}}
\RightLabel{\fns $\topl$}
\UI$\ATOPBOT \fCenter \topl$
\DisplayProof
 \\
 & & & & \\
 &
\AX$\ATOPBOT \fCenter X$
\LeftLabel{\fns $\TOPBOTL$-$W$}
\UI$Y \fCenter X$
\DisplayProof
 & &
\AXC{\phantom{$ \fCenter \top$}}
\LeftLabel{\fns $\botl$}
\UI$\botl \fCenter \ATOPBOT$
\DisplayProof
 &
\AX$X \fCenter \ATOPBOT$
\RightLabel{\fns $\botl$}
\UI$X \fCenter \botl$
\DisplayProof
 \\
\end{tabular}
\end{center}

\begin{itemize}
\item Operational rules for multi-type connectives:
\end{itemize}
\begin{center}
\begin{tabular}{rlrl}


\mc{2}{c}{$\mathsf{L} \to \mathsf{P}^{\mathrm{op}}$} & \mc{2}{c}{$\mathsf{P}^{\mathrm{op}} \to \mathsf{L}$}\\

 & \\
\AX$\WCIRCW A \fCenter \Pi$
\LeftLabel{\fns $\wdiaw$}
\UI$\wdiaw A \fCenter \Pi$
\DisplayProof
 &
\AX$X \fCenter A$
\RightLabel{\fns $\wdiaw$}
\UI$\WCIRCW X \fCenter \wdiaw A$
\DisplayProof
 &
\AX$X \fCenter \BCIRCB \xi$
\LeftLabel{\fns $\bboxb$}
\UI$X \fCenter \bboxb \xi$
\DisplayProof
 &
\AX$\xi \fCenter \Pi$
\RightLabel{\fns $\bboxb$}
\UI$\bboxb \xi \fCenter \BCIRCB \Pi$
\DisplayProof
 \\

 & & & \\

\mc{2}{c}{$\mathsf{P} \to \mathsf{L}$} & \mc{2}{c}{$\mathsf{L} \to \mathsf{P}$} \\

 & \\
\AX$\BCIRCB \alpha \fCenter X$
\LeftLabel{\fns $\bdiab$}
\UI$\bdiab \alpha \fCenter X$
\DisplayProof
 &
\AX$\Gamma \fCenter \alpha$
\RightLabel{\fns $\bdiab$}
\UI$\BCIRCB \Gamma \fCenter \bdiab \alpha$
\DisplayProof
 &
\AX$\Gamma \fCenter \WCIRCW A$
\LeftLabel{\fns $\wboxw$}
\UI$\Gamma \fCenter \wboxw A$
\DisplayProof
 &
\AX$A \fCenter X$
\RightLabel{\fns $\wboxw$}
\UI$\wboxw A \fCenter \WCIRCW X$
\DisplayProof
 \\

\end{tabular}
\end{center}

\section{Properties}\label{sec:properties}
\subsection{Soundness}\label{ssec:soundness}
In the present subsection, we outline the  verification of the soundness of the rules of D.LL w.r.t.~the semantics of  heterogeneous algebras  introduced in Section \ref{sec:heterogeneous algebras for complete lattices}. The first step consists in interpreting structural symbols as logical symbols according to their (precedent or succedent) position,\footnote{\label{footnote:def precedent succedent pos}For any  sequent $x\vdash y$, we define the signed generation trees $+x$ and $-y$ by labelling the root of the generation tree of $x$ (resp.\ $y$) with the sign $+$ (resp.\ $-$), and then propagating the sign to all nodes according to the polarity of the coordinate of the connective assigned to each node. Positive (resp.\ negative) coordinates propagate the same (resp.\ opposite) sign to the corresponding child node.  Then, a substructure $z$ in $x\vdash y$ is in {\em precedent} (resp.\ {\em succedent}) {\em position} if the sign of its root node as a subtree of $+x$ or $-y$ is  $+$ (resp.\ $-$).}
as indicated in the synoptic tables of Section \ref{ssec:language of DLL}. This makes it possible to interpret sequents as inequalities, and rules as quasi-inequalities.
The verification of the soundness of the rules of D.LL  then consists in verifying the validity of their corresponding quasi-inequalities in heterogeneous algebras. The verification of the soundness of pure-type rules and of the introduction rules following this procedure is routine, and is omitted. The only multi-type rules of D.LL are the display rules, the validity of which follows straightforwardly from the adjunctions between  the interpretations of the multi-type connectives involved.

\subsection{Conservativity}
\label{ssec:conservativity}
To argue that the calculus D.LL introduced in Section \ref{MultiTypeDisplayCalculus} adequately captures lattice logic, we follow the standard proof strategy discussed in \cite{GMPTZ,LORI}. Let 
$\models_{\mathrm{HA}}$ denote the semantic consequence relation arising from the heterogeneous algebras introduced in Section \ref{sec:heterogeneous algebras for complete lattices}. We need to show that, for all formulas $A$ and $B$ of the original language of lattice logic, if $A^\tau\vdash B_\tau$ is a D.LL-derivable sequent, then  $A\vdash B$ is a theorem of the Hilbert-style presentation of lattice logic. This claim can be proved using  the following  facts: (a) the rules of D.LL are sound w.r.t.~heterogeneous algebras  (cf.~Section \ref{ssec:soundness}),  (b)  lattice logic is strongly complete w.r.t.~the class of complete lattices, and (c) complete lattices are equivalently presented as   heterogeneous algebras (cf.~Section \ref{sec:heterogeneous algebras for complete lattices}), so that the semantic consequence relation arising from each type of structures preserves and reflects the translation (cf.~Proposition \ref{prop:consequence preserved and reflected}). Then, let $A, B$ be formulas of  the original lattice logic language. If  $A^\tau\vdash B_\tau$ is a D.LL-derivable sequent, then, by (a),  $A^\tau\models_{\mathrm{HA}} B_\tau$. By (c), this implies that $A\models_{\mathrm{LL}} B$, where $\models_{\mathrm{LL}}$ denotes the semantic consequence relation arising from (complete) lattices. By (b), this implies that $A\vdash B$ is a theorem of the Hilbert-style presentation of lattice logic, as required.

\subsection{Cut elimination and subformula property}
\label{ssec:CutEliminationAndSubformulaProperty}

In the present section, we outline the proof of cut elimination and subformula property for the calculus D.LL introduced in Section \ref{MultiTypeDisplayCalculus}. As discussed earlier on, the cut elimination and subformula property  do not need to be proved via the original argument by Gentzen, but can rather be inferred from a meta-theorem, following the strategy introduced by Belnap for display calculi. The meta-theorem to which we will appeal for D.LL was proved in \cite{Trends}, and  in \cite[Theorem A.2]{GP:linear} a restricted version of it is stated, which  specifically applies to {\em proper multi-type display calculi} (cf.~\cite[Definition A.1]{GP:linear}).

By \cite[Theorem A.2]{GP:linear}, it is enough to verify that D.LL is a  proper multi-type display calculus, i.e.~it meets the conditions C$_1$-C$_8$ listed in \cite[Definition A.1]{GP:linear}. All conditions except C$_8$ are readily satisfied by inspecting the rules. In what follows we verify C$_8$. This requires to check that reduction steps are available for every application of the cut rule in which both cut-formulas are principal, which either remove the original cut altogether or replace it by one or more cuts on formulas of strictly lower complexity.

\paragraph*{Atomic propositions:}

\begin{center}
\begin{tabular}{ccc}
\bottomAlignProof
\AX$p \fCenter p$
\AX$p \fCenter p$
\BI$p \fCenter p$
\DisplayProof

 & $\rightsquigarrow$ &

\bottomAlignProof
\AX$p \fCenter p$
\DisplayProof
 \\
\end{tabular}
\end{center}

\paragraph*{Constants:}

\begin{center}
\begin{tabular}{ccc}
\bottomAlignProof
\AX$\ATOPBOT \fCenter \topl$
\AXC{\ \ \ $\vdots$ \raisebox{1mm}{$\pi_1$}}
\noLine
\UI$\ATOPBOT \fCenter X$
\UI$\topl \fCenter X$
\BI$\ATOPBOT \fCenter X$
\DisplayProof

 & $\rightsquigarrow$ &

\bottomAlignProof
\AXC{\ \ \ $\vdots$ \raisebox{1mm}{$\pi_1$}}
\noLine
\UI$\ATOPBOT \fCenter X$
\DisplayProof
 \\
\end{tabular}
\end{center}

\noindent The case for $\botl$ is similar to the one above.

\paragraph*{Binary connectives:}

\begin{center}
\begin{tabular}{ccc}
\!\!\!\!\!
\bottomAlignProof
\AXC{\ \ \ $\vdots$ \raisebox{1mm}{$\pi_1$}}
\noLine
\UI$S \fCenter s$
\AXC{\ \ \ $\vdots$ \raisebox{1mm}{$\pi_2$}}
\noLine
\UI$T \fCenter t$
\BI$S\PANDORP T \fCenter s \pandp t$
\AXC{\ \ \ $\vdots$ \raisebox{1mm}{$\pi_3$}}
\noLine
\UI$s \PANDORP t \fCenter U$
\UI$s \pandp t \fCenter U$
\BI$S \PANDORP T \fCenter U$
\DisplayProof

 & $\rightsquigarrow$ &

\!\!\!\!\!\!\!\!\!\!\!\!\!\!\!\!\!\!\!\!
\bottomAlignProof
\AXC{\ \ \ $\vdots$ \raisebox{1mm}{$\pi_1$}}
\noLine
\UI$S \fCenter s$
\AXC{\ \ \ $\vdots$ \raisebox{1mm}{$\pi_2$}}
\noLine
\UI$T \fCenter t$
\AXC{\ \ \ $\vdots$ \raisebox{1mm}{$\pi_3$}}
\noLine
\UI$s \PANDORP t \fCenter U$

\UI$t \fCenter s \PARRP U$
\BI$T \fCenter s \PARRP U$
\UI$s \PANDORP T \fCenter U$
\UI$T \PANDORP s \fCenter U$
\UI$s \fCenter T \PARRP U$
\BI$S \fCenter T \PARRP U$
\UI$ T \PANDORP S \fCenter U$
\UI$S \PANDORP T \fCenter U$
\DisplayProof
 \\
\end{tabular}
\end{center}

\noindent The case for $s \porp t$ is similar to the one above.

\paragraph*{Multi-type connectives:}

\begin{center}
\begin{tabular}{ccc}
\bottomAlignProof
\AXC{\ \ \ $\vdots$ \raisebox{1mm}{$\pi_1$}}
\noLine
\UI$X \fCenter A$
\UI$\WCIRCW X \fCenter \wdiaw A$
\AXC{\ \ \ $\vdots$ \raisebox{1mm}{$\pi_2$}}
\noLine
\UI$\WCIRCW A \fCenter \Pi$
\UI$\wdiaw A \fCenter \Pi$
\BI$\WCIRCW X \fCenter \Pi$
\DisplayProof

 & $\rightsquigarrow$ &

\!\!\!\!\!\!\!
\bottomAlignProof
\AXC{\ \ \ $\vdots$ \raisebox{1mm}{$\pi_1$}}
\noLine
\UI$X \fCenter A$
\AXC{\ \ \ $\vdots$ \raisebox{1mm}{$\pi_2$}}
\noLine
\UI$\WCIRCW A \fCenter \Pi$
\UI$A \fCenter \BCIRCB \Pi$
\BI$X \fCenter \BCIRCB \Pi$
\UI$\WCIRCW X \fCenter \Pi$
\DisplayProof
 \\
\end{tabular}
\end{center}

\noindent The cases for $\wboxw A$, $\bdiab \alpha$, and $\bboxb \xi$ are   similar to the one above.

\subsection{Completeness}

In order to translate sequents of the original language of lattice logic into sequents in the multi-type language of lattice logic, we will make use of the translations $\tau_1, \tau_2: \mathcal{L}\to \mathcal{L}_{\mathrm{MT}}$ so that for all $A, B\in \mathcal{L}$ and $A\vdash B$, we write
\begin{center}
$\tau_1(A) \vdash \tau_2(B)$\ \ \ \ \ abbreviated as\ \ \ \ \ $A^\tau \vdash B_\tau$.
\end{center}
The translations $\tau_1$ and $\tau_2$ are defined by simultaneous induction as follows:

\begin{center}
\begin{tabular}{rl|rl}

$\topl^\tau ::=$ & $\pbdia \pwbox \, \topl$ & $\topl_\tau ::=$ & $\boxbp \diawp \, \topl$ \\
$\botl^\tau ::=$ & $\pbdia \pwbox \, \botl$ & $\botl_\tau ::=$ & $\boxbp \diawp \, \botl$ \\
$p^\tau      ::=$ & $\pbdia \pwbox \, p$ & $p_\tau ::=$ & $\boxbp \diawp \, p$       \\
$(A \andl B)^\tau ::=$ & $\pbdia (\pwbox \, A^\tau \pand \pwbox \, B^\tau)$             & $(A \andl B)_\tau ::=$ & $ \boxbp (\diawp \, A_\tau \andp \diawp \, B_\tau)$ \\
$(A \orl B)^\tau    ::=$ & $\pbdia (\pwbox \, A^\tau \por \pwbox \, B^\tau)$ & $(A \orl B)_\tau ::=$ & $\boxbp (\diawp \, A_\tau \orp \diawp \, B_\tau)$ \\
\end{tabular}
\end{center}

\begin{proposition}
\label{prop: identity theorem via translation}
For every $A\in \mathcal{L}$, the multi-type sequent $A^\tau \vdash A_\tau$ is derivable in D.LL.
\end{proposition}

\begin{proof}
By simultaneous induction on $A\in \mathsf{L}$, $\alpha\in \mathsf{P}$, and $\xi\in \mathsf{P^{op}}$. 

\begin{itemize}
\item Base cases: $A := \topl$, $A := \botl$ and $A := p$

\begin{center}
\begin{tabular}{ccc}
\AXC{\phantom{$ \fCenter \top$}}
\RightLabel{\fns $\topl$}
\UI$\ATOPBOT \fCenter \topl$
\UI$\topl \fCenter \topl$
\UI$\pwboxwp \topl \fCenter \PWCIRCWP \topl$
\UI$\PBCIRCBP \pwboxwp \topl \fCenter \topl$
\UI$\pbdiabp \pwboxwp \topl \fCenter \topl$
\UI$\PWCIRCWP \pbdiabp \pwboxwp \topl \fCenter \pwdiawp \topl$
\UI$\pbdiabp \pwboxwp \topl \fCenter \PBCIRCBP \pwdiawp \topl$
\UI$\pbdiabp \pwboxwp \topl \fCenter \pbboxbp \pwdiawp \topl$
\DisplayProof
 &
\AXC{\phantom{$\bot \fCenter $}}
\LeftLabel{\fns $\botl$}
\UI$\botl \fCenter \ATOPBOT$
\UI$\botl \fCenter \botl$
\UI$\pwboxwp \botl \fCenter \PWCIRCWP \botl$
\UI$\PBCIRCBP \pwboxwp \botl \fCenter \botl$
\UI$\pbdiabp \pwboxwp \botl \fCenter \botl$
\UI$\PWCIRCWP \pbdiabp \pwboxwp \botl \fCenter \pwdiawp \botl$
\UI$\pbdiabp \pwboxwp \botl \fCenter \PBCIRCBP \pwdiawp \botl$
\UI$\pbdiabp \pwboxwp \botl \fCenter \pbboxbp \pwdiawp \botl$
\DisplayProof
 &
\AXC{\phantom{$p \fCenter p$}}
\LeftLabel{\fns $Id$}
\UI$p \fCenter p$
\UI$\pwboxwp p \fCenter \PWCIRCWP p$
\UI$\PBCIRCBP \pwboxwp p \fCenter p$
\UI$\pbdiabp \pwboxwp p \fCenter p$
\UI$\PWCIRCWP \pbdiabp \pwboxwp p \fCenter \pwdiawp p$
\UI$\pbdiabp \pwboxwp p \fCenter \PBCIRCBP \pwdiawp p$
\UI$\pbdiabp \pwboxwp p \fCenter \pbboxbp \pwdiawp p$
\DisplayProof
\end{tabular}
\end{center}

\item Inductive case: $A = B \andl C$

\begin{center}
\begin{tabular}{c}
\AXC{\phantom{$ \fCenter $}}
\LeftLabel{\fns ind.~hyp.}
\UI$B^\tau \fCenter B_\tau$
\UI$\pwboxwp B^\tau \fCenter \PWCIRCWP B_\tau$
\LeftLabel{\fns $W$}
\UI$\pwboxwp B^\tau \PANDORP \pwboxwp C^\tau \fCenter \PWCIRCWP B_\tau$
\UI$\pwboxwp B^\tau \pandp \pwboxwp C^\tau \fCenter \PWCIRCWP B_\tau$
\UI$\PBCIRCBP \pwboxwp B^\tau \pandp \pwboxwp C^\tau \fCenter B_\tau$
\UI$\pbdiabp (\pwboxwp B^\tau \pandp \pwboxwp C^\tau) \fCenter B_\tau$
\UI$\PWCIRCWP \pbdiabp (\pwboxwp B^\tau \pandp \pwboxwp C^\tau) \fCenter \pwdiawp B_\tau$
\AXC{\phantom{$ \fCenter $}}
\LeftLabel{\fns ind.~hyp.}
\UI$C^\tau \fCenter C_\tau$
\UI$\pwboxwp C^\tau \fCenter \PWCIRCWP C_\tau$
\LeftLabel{\fns $W$}
\UI$\pwboxwp C^\tau \PANDORP \pwboxwp B^\tau \fCenter \PWCIRCWP C_\tau$
\LeftLabel{\fns $E$}
\UI$\pwboxwp B^\tau \PANDORP \pwboxwp C^\tau \fCenter \PWCIRCWP C_\tau$
\UI$\pwboxwp B^\tau \pandp \pwboxwp C^\tau \fCenter \PWCIRCWP C_\tau$
\UI$\PBCIRCBP \pwboxwp B^\tau \pandp \pwboxwp C^\tau \fCenter C_\tau$
\UI$\pbdiabp (\pwboxwp B^\tau \pandp \pwboxwp C^\tau) \fCenter C_\tau$
\UI$\PWCIRCWP \pbdiabp (\pwboxwp B^\tau \pandp \pwboxwp C^\tau) \fCenter \pwdiawp C_\tau$
\BI$\PWCIRCWP \pbdiabp (\pwboxwp B^\tau \pandp \pwboxwp C^\tau) \PANDORP \PWCIRCWP \pbdiabp (\pwboxwp B^\tau \pandp \pwboxwp C^\tau) \fCenter \pwdiawp B_\tau \pandp \pwdiawp C_\tau$
\LeftLabel{\fns $C$}
\UI$\PWCIRCWP \pbdiabp (\pwboxwp B^\tau \pandp \pwboxwp C^\tau) \fCenter \pwdiawp B_\tau \pandp \pwdiawp C_\tau$
\UI$\pbdiabp (\pwboxwp B^\tau \pandp \pwboxwp C^\tau) \fCenter \PBCIRCBP \pwdiawp B_\tau \pandp \pwdiawp C_\tau$
\UI$\pbdiabp (\pwboxwp B^\tau \pandp \pwboxwp C^\tau) \fCenter \pbboxbp (\pwdiawp B_\tau \pandp \pwdiawp C_\tau)$
\DisplayProof
 \\
\end{tabular}
\end{center}
The case in which  $A = B \orl C$ is derived symmetrically.

\end{itemize}
\end{proof}

In what follows, we only derive the translations of the axioms involving conjunction, since the axioms involving disjunction can be treated symmetrically.


\begin{flushleft}
\begin{tabular}{rlcl}
         & Commutative laws                  &                               & translation \\
\hline
cC1. & $(A \andl B)^\tau \vdash (B \andl A)_\tau$ & $\rightsquigarrow$ & $\pbdiabp (\pwboxwp A^\tau \pandp \pwboxwp B^\tau) \vdash \pbboxbp (\pwdiawp B_\tau \pandp \pwdiawp A_\tau)$ \\
cC2. & $(B \andl A)^\tau \vdash (A \andl B)_\tau$ & $\rightsquigarrow$ & $\pbdiabp (\pwboxwp B^\tau \pandp \pwboxwp A^\tau) \vdash \pbboxbp (\pwdiawp A_\tau \pandp \pwdiawp B_\tau)$ \\
\end{tabular}
\end{flushleft}

\noindent Although each connective in succedent position should have the superscript $^\mathrm{op}$, in what follows, for the sake of readability, we suppress it both in the translations and in the derivation trees of the axioms.

\begin{center}
\begin{tabular}{c}
\AX$B^\tau \fCenter B_\tau$
\UI$\pwboxwp B^\tau \fCenter \PWCIRCWP B_\tau$
\LeftLabel{\fns $W$}
\UI$\pwboxwp B^\tau \PANDORP \pwboxwp A^\tau \fCenter \PWCIRCWP B_\tau$
\LeftLabel{\fns $E$}
\UI$\pwboxwp A^\tau \PANDORP \pwboxwp B^\tau \fCenter \PWCIRCWP B_\tau$
\UI$\pwboxwp A^\tau \pandp \pwboxwp B^\tau \fCenter \PWCIRCWP B_\tau$
\UI$\PBCIRCBP \pwboxwp A^\tau \pandp \pwboxwp B^\tau \fCenter B_\tau$
\UI$\pbdiabp (\pwboxwp A^\tau \pandp \pwboxwp B^\tau) \fCenter B_\tau$
\UI$\PWCIRCWP \pbdiabp (\pwboxwp A^\tau \pandp \pwboxwp B^\tau) \fCenter \pwdiawp B_\tau$

\AX$A^\tau \fCenter A_\tau$
\UI$\pwboxwp A^\tau \fCenter \PWCIRCWP A_\tau$
\LeftLabel{\fns $W$}
\UI$\pwboxwp A^\tau \PANDORP \pwboxwp B^\tau \fCenter \PWCIRCWP A_\tau$
\UI$\pwboxwp A^\tau \pandp \pwboxwp B^\tau \fCenter \PWCIRCWP A_\tau$
\UI$\PBCIRCBP \pwboxwp A^\tau \pandp \pwboxwp B^\tau \fCenter A_\tau$
\UI$\pbdiabp (\pwboxwp A^\tau \pandp \pwboxwp B^\tau) \fCenter A_\tau$
\UI$\PWCIRCWP \pbdiabp (\pwboxwp A^\tau \pandp \pwboxwp B^\tau) \fCenter \pwdiawp A_\tau$

\BI$\PWCIRCWP \pbdiabp (\pwboxwp A^\tau \pandp \pwboxwp B^\tau) \PANDORP \PWCIRCWP \pbdiabp (\pwboxwp A^\tau \pandp \pwboxwp B^\tau) \fCenter \pwdiawp B_\tau \pandp \pwdiawp A_\tau$
\LeftLabel{\fns $C$}
\UI$\PWCIRCWP \pbdiabp (\pwboxwp A^\tau \pandp \pwboxwp B^\tau) \fCenter \pwdiawp B_\tau \pandp \pwdiawp A_\tau$
\UI$\pbdiabp (\pwboxwp A^\tau \pandp \pwboxwp B^\tau) \fCenter \PBCIRCBP \pwdiawp B_\tau \pandp \pwdiawp A_\tau$
\UI$\pbdiabp (\pwboxwp A^\tau \pandp \pwboxwp B^\tau) \fCenter \pbboxbp (\pwdiawp B_\tau \pandp \pwdiawp A_\tau)$
\DisplayProof
 \\
\end{tabular}
\end{center}

\newpage

\begin{flushleft}
\begin{tabular}{rlcl}
           & Associative laws                                                 &                                & translation \\
\hline
cA1.   & $(A \andl (B \andl C))^\tau \vdash ((A \andl B) \andl C)_\tau$ & $\rightsquigarrow$ & \\
 &
\mc{3}{c}{\ \ \ \ \ \ \ \ \ \ \ \ \ \ \ \ \ \ \ \ \ \ \ \ \ $\pbdiabp (\pwboxwp A^\tau \pandp \pwboxwp \pbdiabp (\pwboxwp B^\tau \pandp \pwboxwp C^\tau)) \vdash \pbboxbp (\pwdiawp \pbboxbp (\pwdiawp A_\tau \pandp \pwdiawp B_\tau) \pandp \pwdiawp C_\tau)$} \\
cA2.   & $((A \andl B) \andl C)^\tau \vdash (A \andl (B \andl C))_\tau$ & $\rightsquigarrow$ & \\
 &
\mc{3}{c}{\ \ \ \ \ \ \ \ \ \ \ \ \ \ \ \ \ \ \ \ \ \ \ \ \ $\pbdiabp (\pwboxwp \pbdiabp (\pwboxwp A^\tau \pandp \pwboxwp B^\tau) \pandp \pwboxwp C^\tau) \vdash \pbboxbp (\pwdiawp A_\tau \pandp \pwdiawp \pbboxbp (\pwdiawp B_\tau \pandp \pwdiawp C_\tau))$} \\
\end{tabular}
\end{flushleft}

\noindent Although each formula variable in precedent (resp.~succedent) position should be written with the superscript $^\tau$ (resp.~subscript $_\tau$), in what follows, for the sake of readability, we suppress it in the derivation trees of the axioms.

\begin{center}
{\fns
%
 }
\end{center}

\section{Distributivity fails}\label{sec:distrib fails}
In the present section, we show that the translation of the distributivity axiom is not derivable in D.LL.

\begin{flushleft}
\begin{tabular}{rlcl}
          & Distributivity laws                     &                                & translation \\
\hline
cD1. & $(A \pandp (B \porp C))^\tau \fCenter ((A \pandp B) \porp (A \porp C))_\tau$ & $\rightsquigarrow$ & \\
 &
\mc{3}{c}{$\pbdiabp \Big(\pwboxwp A^\tau \pandp \pwboxwp \pbdiabp (\pwboxwp A^\tau \porp \pwboxwp B^\tau)\Big) \fCenter \pbboxbp \Big( \pwdiawp \pbboxbp (\pwdiawp A_\tau \pandp \pwdiawp B_\tau) \porp \pwdiawp \pbboxbp (\pwdiawp A_\tau \pandp \pwdiawp C_\tau) \Big)$} \\
\end{tabular}
\end{flushleft}

Our strategy will be to show that all the possible paths in the backward proof search always end in  deadlocks.
First,  we apply exhaustively backward all  invertible operational rules (modulo applications of display postulates):

\begin{center}
\begin{tabular}{c}
\AXC{$\!\!\wn\wn\wn$}
\noLine
\UIC{$\!\!\vdots$}
\noLine
\UI$\PBCIRCBP \Big( \pwboxwp A \PANDORP \pwboxwp \pbdiabp (\pwboxwp B \porp \pwboxwp C) \Big) \fCenter \PBCIRCBP \Big( \pwdiawp \pbboxbp (\pwdiawp A \pandp \pwdiawp B) \PANDORP \pwdiawp \pbboxbp (\pwdiawp A \pandp \pwdiawp C) \Big)$

\UI$\pwboxwp A \PANDORP \pwboxwp \pbdiabp (\pwboxwp B \porp \pwboxwp C) \fCenter \PWCIRCWP \PBCIRCBP \Big( \pwdiawp \pbboxbp (\pwdiawp A \pandp \pwdiawp B) \PANDORP \pwdiawp \pbboxbp (\pwdiawp A \pandp \pwdiawp C) \Big)$

\UI$\pwboxwp A \pandp \pwboxwp \pbdiabp (\pwboxwp B \porp \pwboxwp C) \fCenter \PWCIRCWP \PBCIRCBP \Big( \pwdiawp \pbboxbp (\pwdiawp A \pandp \pwdiawp B) \PANDORP \pwdiawp \pbboxbp (\pwdiawp A \pandp \pwdiawp C) \Big)$

\UI$\PBCIRCBP \Big(\pwboxwp A \pandp \pwboxwp \pbdiabp (\pwboxwp B \porp \pwboxwp C) \Big) \fCenter \PBCIRCBP \Big( \pwdiawp \pbboxbp (\pwdiawp A \pandp \pwdiawp B) \PANDORP \pwdiawp \pbboxbp (\pwdiawp A \pandp \pwdiawp C) \Big)$

\UI$\PWCIRCWP \PBCIRCBP \Big(\pwboxwp A \pandp \pwboxwp \pbdiabp (\pwboxwp B \porp \pwboxwp C) \Big) \fCenter \pwdiawp \pbboxbp (\pwdiawp A \pandp \pwdiawp B) \PANDORP \pwdiawp \pbboxbp (\pwdiawp A \pandp \pwdiawp C)$

\UI$\PWCIRCWP \PBCIRCBP \Big(\pwboxwp A \pandp \pwboxwp \pbdiabp (\pwboxwp B \porp \pwboxwp C) \Big) \fCenter \pwdiawp \pbboxbp (\pwdiawp A \pandp \pwdiawp B) \porp \pwdiawp \pbboxbp (\pwdiawp A \pandp \pwdiawp C)$

\UI$\PBCIRCBP \Big(\pwboxwp A \pandp \pwboxwp \pbdiabp (\pwboxwp B \porp \pwboxwp C) \Big) \fCenter \PBCIRCBP \Big( \pwdiawp \pbboxbp (\pwdiawp A \pandp \pwdiawp B) \porp \pwdiawp \pbboxbp (\pwdiawp A \pandp \pwdiawp C) \Big)$

\UI$\PBCIRCBP \Big(\pwboxwp A \pandp \pwboxwp \pbdiabp (\pwboxwp B \porp \pwboxwp C) \Big) \fCenter \pbboxbp \Big( \pwdiawp \pbboxbp (\pwdiawp A \pandp \pwdiawp B) \porp \pwdiawp \pbboxbp (\pwdiawp A \pandp \pwdiawp C) \Big)$

\UI$\pbdiabp \Big(\pwboxwp A \pandp \pwboxwp \pbdiabp (\pwboxwp B \porp \pwboxwp C) \Big) \fCenter \pbboxbp \Big( \pwdiawp \pbboxbp (\pwdiawp A \pandp \pwdiawp B) \porp \pwdiawp \pbboxbp (\pwdiawp A \pandp \pwdiawp C) \Big)$
\DisplayProof
 \\
\end{tabular}
\end{center}

There are no  structural rules in which  $\PBCIRCBP$ and $\PANDOR$ interact, therefore we are reduced to the following possibilities:  either we isolate the structure $$X=\pwboxwp A \PANDORP \pwboxwp \pbdiabp (\pwboxwp B \porp \pwboxwp C)$$ in precedent position by means of a backward application of a display postulate, or we similarly isolate the structure $$Y=\pwdiawp \pbboxbp (\pwdiawp A \pandp \pwdiawp B) \PANDORP \pwdiawp \pbboxbp (\pwdiawp A \pandp \pwdiawp C)$$ in succedent position.

In what follows, we  treat the first case, since the argument for the second case is analogous. Once  the structure $X$ is in isolation, we can act on $X$ only via Exchange, Weakening or Residuation. However, each of these moves  will lead us to a dead end, as we show below.

\begin{itemize}
\item Case 1: (Exchange or) Residuation.
\end{itemize}

As an intermediate step, we can try to isolate any of the substructures of $X$ via Residuation, as follows:

\begin{center}
\begin{tabular}{c}
\AXC{$\!\!\wn\wn\wn$}
\noLine
\UIC{$\!\!\vdots$}
\noLine

\UI$\pwboxwp \pbdiabp (\pwboxwp B \porp \pwboxwp C) \fCenter \pwboxwp A \PARRP \PWCIRCWP \PBCIRCBP \Big( \pwdiawp \pbboxbp (\pwdiawp A \pandp \pwdiawp B) \PANDORP \pwdiawp \pbboxbp (\pwdiawp A \pandp \pwdiawp C) \Big)$

\UI$\pwboxwp A \PANDORP \pwboxwp \pbdiabp (\pwboxwp B \porp \pwboxwp C) \fCenter \PWCIRCWP \PBCIRCBP \Big( \pwdiawp \pbboxbp (\pwdiawp A \pandp \pwdiawp B) \PANDORP \pwdiawp \pbboxbp (\pwdiawp A \pandp \pwdiawp C) \Big)$

\DisplayProof
 \\
\end{tabular}
\end{center}

or via Exchange and Residuation, as follows:

\begin{center}
\begin{tabular}{c}
\AXC{$\!\!\wn\wn\wn$}
\noLine
\UIC{$\!\!\vdots$}
\noLine

\UI$\pwboxwp A \fCenter \pwboxwp \pbdiabp (\pwboxwp B \porp \pwboxwp C) \PARRP \PWCIRCWP \PBCIRCBP \Big( \pwdiawp \pbboxbp (\pwdiawp A \pandp \pwdiawp B) \PANDORP \pwdiawp \pbboxbp (\pwdiawp A \pandp \pwdiawp C) \Big)$

\UI$\pwboxwp \pbdiabp (\pwboxwp B \porp \pwboxwp C) \PANDORP \pwboxwp A \fCenter \PWCIRCWP \PBCIRCBP \Big( \pwdiawp \pbboxbp (\pwdiawp A \pandp \pwdiawp B) \PANDORP \pwdiawp \pbboxbp (\pwdiawp A \pandp \pwdiawp C) \Big)$

\UI$\pwboxwp A \PANDORP \pwboxwp \pbdiabp (\pwboxwp B \porp \pwboxwp C) \fCenter \PWCIRCWP \PBCIRCBP \Big( \pwdiawp \pbboxbp (\pwdiawp A \pandp \pwdiawp B) \PANDORP \pwdiawp \pbboxbp (\pwdiawp A \pandp \pwdiawp C) \Big)$

\DisplayProof
 \\
\end{tabular}
\end{center}

\noindent However, in each case we  reach a dead end.

\begin{itemize}
\item Case 2: (Exchange or) Weakening.
\end{itemize}

As an intermediate step, we can try to isolate an immediate substructure of $X$ by applying backward Weakening. By directly applying Weakening, we obtain $$\pwboxwp A \fCenter \PWCIRCWP \PBCIRCBP \Big( \pwdiawp \pbboxbp (\pwdiawp A \pandp \pwdiawp B) \PANDORP \pwdiawp \pbboxbp (\pwdiawp A \pandp \pwdiawp C) \Big),$$ and by applying Exchange and Weakening, we obtain $$\pwboxwp \pbdiabp (\pwboxwp B \porp \pwboxwp C)  \fCenter \PWCIRCWP \PBCIRCBP \Big( \pwdiawp \pbboxbp (\pwdiawp A \pandp \pwdiawp B) \PANDORP \pwdiawp \pbboxbp (\pwdiawp A \pandp \pwdiawp C) \Big).$$

In each subcase, this choice leads us to a dead end. Indeed, we preliminarily observe that the second subcase can be reduced to the first one by expanding the tree as follows:

\begin{center}
\begin{tabular}{c}
\!\!\!\!\!\!\!\!\!\!\!\!\!\!\!\!\!\!\!\!\!\!\!\!\!\!\!\!
\AXC{$\!\!\wn\wn$}
\noLine
\UIC{$\!\!\vdots$}
\noLine

\UI$\pwboxwp B\fCenter \PWCIRCWP \PBCIRCBP \Big( \pwdiawp \pbboxbp (\pwdiawp A \pandp \pwdiawp B) \PANDORP \pwdiawp \pbboxbp (\pwdiawp A \pandp \pwdiawp C) \Big)$

\AXC{$\!\!\wn\wn$}
\noLine
\UIC{$\!\!\vdots$}
\noLine

\UI$\pwboxwp C\fCenter \PWCIRCWP \PBCIRCBP \Big( \pwdiawp \pbboxbp (\pwdiawp A \pandp \pwdiawp B) \PANDORP \pwdiawp \pbboxbp (\pwdiawp A \pandp \pwdiawp C) \Big)$

\BI$\pwboxwp B \porp \pwboxwp C \fCenter \PWCIRCWP \PBCIRCBP \Big( \pwdiawp \pbboxbp (\pwdiawp A \pandp \pwdiawp B) \PANDORP \pwdiawp \pbboxbp (\pwdiawp A \pandp \pwdiawp C) \Big) \PANDORP \PWCIRCWP \PBCIRCBP \Big( \pwdiawp \pbboxbp (\pwdiawp A \pandp \pwdiawp B) \PANDORP \pwdiawp \pbboxbp (\pwdiawp A \pandp \pwdiawp C) \Big)$

\UI$\pwboxwp B \porp \pwboxwp C \fCenter \PWCIRCWP \PBCIRCBP \Big( \pwdiawp \pbboxbp (\pwdiawp A \pandp \pwdiawp B) \PANDORP \pwdiawp \pbboxbp (\pwdiawp A \pandp \pwdiawp C) \Big)$

\UI$\PBCIRCBP (\pwboxwp B \porp \pwboxwp C) \Big) \fCenter \PBCIRCBP \Big( \pwdiawp \pbboxbp (\pwdiawp A \pandp \pwdiawp B) \PANDORP \pwdiawp \pbboxbp (\pwdiawp A \pandp \pwdiawp C) \Big)$

\UI$\pbdiabp (\pwboxwp B \porp \pwboxwp C) \Big) \fCenter \PBCIRCBP \Big( \pwdiawp \pbboxbp (\pwdiawp A \pandp \pwdiawp B) \PANDORP \pwdiawp \pbboxbp (\pwdiawp A \pandp \pwdiawp C) \Big)$

\UI$\pbdiabp (\pwboxwp B \porp \pwboxwp C) \Big) \fCenter \PBCIRCBP \Big( \pwdiawp \pbboxbp (\pwdiawp A \pandp \pwdiawp B) \PANDORP \pwdiawp \pbboxbp (\pwdiawp A \pandp \pwdiawp C) \Big)$

\UI$\pwboxwp \pbdiabp (\pwboxwp B \porp \pwboxwp C) \Big) \fCenter \PWCIRCWP \PBCIRCBP \Big( \pwdiawp \pbboxbp (\pwdiawp A \pandp \pwdiawp B) \PANDORP \pwdiawp \pbboxbp (\pwdiawp A \pandp \pwdiawp C) \Big)$
\DisplayProof
 \\
\end{tabular}
\end{center}

As to the proof of first subcase, let us preliminarily perform the following steps:

\begin{center}
\begin{tabular}{c}
\AXC{$\!\!\wn\wn$}
\noLine
\UIC{$\!\!\vdots$}
\noLine

\UI$\PWCIRCWP A \fCenter \pwdiawp \pbboxbp (\pwdiawp A \pandp \pwdiawp B) \PANDORP \pwdiawp \pbboxbp (\pwdiawp A \pandp \pwdiawp C)$

\UI$A \fCenter \PBCIRCBP \Big( \pwdiawp \pbboxbp (\pwdiawp A \pandp \pwdiawp B) \PANDORP \pwdiawp \pbboxbp (\pwdiawp A \pandp \pwdiawp C) \Big)$

\UI$\pwboxwp A \fCenter \PWCIRCWP \PBCIRCBP \Big( \pwdiawp \pbboxbp (\pwdiawp A \pandp \pwdiawp B) \PANDORP \pwdiawp \pbboxbp (\pwdiawp A \pandp \pwdiawp C) \Big)$
\DisplayProof
 \\
\end{tabular}
\end{center}

Again, we are in a situation in which we can act on the structure $Y$ only via Exchange, Weakening or Residuation, and also in this case any option leads us to a dead end. Indeed:

\begin{itemize}
\item[-] Case 2.1: Exchange or Weakening.

As an intermediate step, we can try to delete one of the immediate substructures of $Y$. By applying Weakening or, respectively, Exchange and Weakening, we obtain $$\PWCIRCWP A \fCenter \pwdiawp \pbboxbp (\pwdiawp A \pandp \pwdiawp B) \quad \textrm{and} \quad \PWCIRCWP A \fCenter \pwdiawp \pbboxbp (\pwdiawp A \pandp \pwdiawp C).$$ In each case, we reach a dead end,  as we show below:

\begin{center}
\begin{tabular}{cc}
\AXC{$\!\!\wn$}
\noLine
\UIC{$\!\!\vdots$}
\noLine

\UI$\PWCIRCWP A \fCenter \pwdiawp A \pandp \pwdiawp B$
\UI$A \fCenter \PBCIRCBP (\pwdiawp A \pandp \pwdiawp B)$
\UI$A \fCenter \pbboxbp (\pwdiawp A \pandp \pwdiawp B)$
\UI$\PWCIRCWP A \fCenter \pwdiawp \pbboxbp (\pwdiawp A \pandp \pwdiawp B)$
\DisplayProof
 &
\AXC{$\!\!\wn$}
\noLine
\UIC{$\!\!\vdots$}
\noLine

\UI$\PWCIRCWP A \fCenter \pwdiawp A \pandp \pwdiawp C$
\UI$A \fCenter \PBCIRCBP (\pwdiawp A \pandp \pwdiawp C)$
\UI$A \fCenter \pbboxbp (\pwdiawp A \pandp \pwdiawp C)$
\UI$\PWCIRCWP A \fCenter \pwdiawp \pbboxbp (\pwdiawp A \pandp \pwdiawp C)$
\DisplayProof

 \\
\end{tabular}
\end{center}

\item[-] Case 2.2: Residuation.
As an intermediate step, we can try to isolate any of the substructures of $Y$ via Residuation, as follows:
\begin{center}
\begin{tabular}{c}
\AXC{$\!\!\wn$}
\noLine
\UIC{$\!\!\vdots$}
\noLine
\UI$\pwdiawp \pbboxbp (\pwdiawp A \pandp \pwdiawp B) \PARRP \PWCIRCWP A \fCenter   \pwdiawp \pbboxbp (\pwdiawp A \pandp \pwdiawp C)$

\UI$\PWCIRCWP A \fCenter \pwdiawp \pbboxbp (\pwdiawp A \pandp \pwdiawp B) \PANDORP \pwdiawp \pbboxbp (\pwdiawp A \pandp \pwdiawp C)$
\DisplayProof
 \\
\end{tabular}
\end{center}
or via Exchange and Residuation, as follows:
\begin{center}
\begin{tabular}{c}
\AXC{$\!\!\wn$}
\noLine
\UIC{$\!\!\vdots$}
\noLine
\UI$\pwdiawp \pbboxbp (\pwdiawp A \pandp \pwdiawp C)\PARRP\PWCIRCWP A \fCenter  \pwdiawp \pbboxbp (\pwdiawp A \pandp \pwdiawp B) $

\UI$\PWCIRCWP A \fCenter \pwdiawp \pbboxbp (\pwdiawp A \pandp \pwdiawp C)\PANDORP \pwdiawp \pbboxbp (\pwdiawp A \pandp \pwdiawp B) $

\UI$\PWCIRCWP A \fCenter \pwdiawp \pbboxbp (\pwdiawp A \pandp \pwdiawp B) \PANDORP \pwdiawp \pbboxbp (\pwdiawp A \pandp \pwdiawp C)$
\DisplayProof
 \\
\end{tabular}
\end{center}

However, in each case we reach a dead end.
\end{itemize}

\end{document}